\documentclass{article}
\usepackage[utf8]{inputenc}
\usepackage[margin=1in]{geometry}
\usepackage{mathtools}
\usepackage{amsfonts}
\usepackage{amsmath}
\usepackage{amssymb}
\usepackage{amsthm}
\usepackage{enumitem}
\usepackage{thmtools}
\usepackage{thm-restate}
\usepackage{caption}

\usepackage{hyperref}
\usepackage[english]{babel}
\usepackage[style=numeric, maxbibnames=99]{biblatex}
\DeclareFieldFormat[misc]{title}{\mkbibquote{#1}}
\usepackage{csquotes}

\hypersetup{
pdfencoding=auto,
psdextra,
breaklinks=true}

\declaretheorem[numberwithin=section]{theorem}
\declaretheorem[numbered=no, name=Theorem]{theorem*}
\declaretheorem[numberlike=theorem]{lemma,proposition,problem,corollary,conjecture}

\usepackage{tikz}
\usepackage{tikz-cd}
\usetikzlibrary{decorations.pathmorphing, decorations.pathreplacing, decorations.shapes}
\usetikzlibrary{nfold}

\tikzset{
dot/.style = {circle, fill, minimum size=#1,
              inner sep=0pt, outer sep=0pt},
dot/.default = 4pt 
}

\tikzset{
cdot/.style = {circle, fill=#1, minimum size=4pt,
              inner sep=0pt, outer sep=0pt},
cdot/.default = black 
}

\renewcommand{\vert}[2]{\node[dot] (#1) at (#2) {};}

\newcommand{\cvert}[3]{\node[cdot={#3}] (#1) at (#2) {};}

\newcommand{\labvert}[3]{\node[dot, label={[label distance=-2pt] #3:{$#1$}}] (#1) at (#2) {};}
\newcommand{\alabvert}[4]{\node[dot, label={[label distance=-2pt] #4:{$#3$}}] (#1) at (#2) {};}

\newcommand{\aclabvert}[5]{\node[cdot={#5}, label={[label distance=-2pt] #4:{$#3$}}] (#1) at (#2) {};}

\newcommand{\edge}[2]{\draw (#1) -- (#2);}
\newcommand{\labdotedge}[4]{\draw[dashed] (#1) -- node[midway, #4] {$#3$} (#2);}
\newcommand{\dilabedge}[4]{\draw[->, shorten >=4pt] (#1) -- node[midway, #4] {$#3$} (#2);}
\newcommand{\urpath}[2]{\draw[decorate,decoration={snake,amplitude=1.5,segment length=10, post length=0,pre length=0}] (#1) -- (#2);}
\newcommand{\laburpath}[4]{\draw[decorate,decoration={snake,amplitude=1.5,segment length=10, post length=0,pre length=0}] (#1) -- node[midway, #4] {$#3$} (#2);}
\newcommand{\spoke}[3]{\draw[-|, shorten >=4pt] (#1) -- ($(#2)!0.5!(#3)$);}
\newcommand{\manspoke}[2]{\draw[-|, shorten >=4pt] (#1) -- (#2);}
\newcommand{\rung}[4]{\draw[|-|, shorten >=4pt, shorten <=4pt] ($(#1)!0.5!(#2)$) -- ($(#3)!0.5!(#4)$);}

\newcommand{\dicoledge}[3]{\draw[->, shorten >=4pt, color=#3] (#1) -- (#2);}
\newcommand{\dicolurpath}[3]{\draw[->, decorate, decoration={snake, amplitude=1.5, segment length=10, post length=1mm, pre length=0}, shorten >=4pt, color=#3] (#1) -- (#2);}

\usetikzlibrary{arrows.meta}
\usetikzlibrary{decorations.markings}
\usetikzlibrary{automata,arrows}
\usetikzlibrary{positioning}
\usetikzlibrary{calc}
\usetikzlibrary{shadings}

\newcommand{\Cf}{\mathfrak{C}}
\newcommand{\Fc}{\mathcal{F}}
\newcommand{\Ff}{\mathfrak{F}}
\newcommand{\Gf}{\mathfrak{G}}
\newcommand{\Hf}{\mathfrak{H}}
\newcommand{\Sc}{\mathcal{S}}
\newcommand{\Sf}{\mathfrak{S}}
\newcommand{\Tc}{\mathcal{T}}
\newcommand{\Z}{\mathbb{Z}}
\newcommand{\ud}{\mathbf{undefined}}

\DeclarePairedDelimiter{\abs}{\lvert}{\rvert}
\DeclareMathOperator{\Py}{Py}
\DeclareMathOperator{\W}{W}
\DeclareMathOperator{\im}{im}
\DeclareMathOperator{\cl}{cl}
\DeclareMathOperator{\conn}{conn}

\title{Pathographs and some (un)decidability results}
\author{Daniel Carter\thanks{Princeton University, Fine Hall, Washington Road, Princeton, NJ 08544-1000, USA}\and Nicolas Trotignon\thanks{CNRS, ENS de Lyon, Universit\'e Claude Bernard Lyon 1, LIP UMR 5668, 69342 Lyon Cedex 07, France}}
\date{}

\begin{document}

\maketitle

\begin{abstract}
We introduce \textit{pathographs} as a framework to study graph classes defined by forbidden structures, including forbidding induced subgraphs, minors, etc. Pathographs approximately generalize s-graphs of L\'ev\^eque--Lin--Maffray--Trotignon by the addition of two extra adjacency relations: one between subdivisible edges and vertices called \textit{spokes}, and one between pairs of subdivisible edges called \textit{rungs}. We consider the following decision problem: given a pathograph $\Hf$ and a finite set of pathographs $\Fc$, is there an $\Fc$-free realization of $\Hf$? This may be regarded as a generalization of the ``graph class containment problem'': given two graph classes $S$ and $S'$, is it the case that $S\subseteq S'$? We prove the pathograph realization problem is undecidable in general, but it is decidable in the case that $\Hf$ has no rungs (but may have spokes), or if $\Fc$ is closed under adding edges, spokes, and rungs. We also discuss some potential applications to proving decomposition theorems.
\end{abstract}

\section{Introduction}

Graphs in this paper are finite and simple. Many questions in graph theory can be reformulated into the following general form:

\begin{problem}\label{prob:general}
Given two classes of graphs $S$ and $S'$ defined by forbidding some graph structures, is it the case that $S\subseteq S'$?

Equivalently, given two sets of graph structures $A$ and $B$, is it true that every graph containing a structure in $A$ necessarily contains a structure in $B$?
\end{problem}

Here, ``graph structures'' can mean many things, including (induced) subgraphs, (induced) minors, (induced) topological minors, and ``Truemper configurations'', among others. We define those containment relations and structures that are less familiar:
\begin{itemize}
    \item A graph $H$ is said to be an \textit{induced minor} of $G$ if $H$ can be obtained from $G$ by a series of vertex deletions and edge contractions (but not edge deletions), or equivalently if there are disjoint connected subgraphs $\{C_v\}_{v\in V(H)}$ of $G$ such that some vertex in $C_v$ is adjacent to some vertex in $C_{v'}$ if and only if $v$ and $v'$ are adjacent in $H$.
    \item A graph $H$ is said to be an \textit{induced topological minor} of $G$ if $H$ can be obtained from $G$ by a series of vertex deletions and replacing vertices of degree 2 by edges, or equivalently if a subdivision of $H$ is an induced subgraph of $G$.
\end{itemize}

We recall the definitions of the four \textit{Truemper configurations}:
\begin{itemize}
    \item A \textit{theta} consists of two vertices $a$ and $b$ and three pairwise disjoint nonadjacent induced $a$-$b$ paths $P_1,P_2,P_3$, all of length at least 2.
    \item A \textit{pyramid} consists of a vertex $a$, triangle $b_1b_2b_3$, and three disjoint nonadjacent induced paths $P_i$ from $a$ to $b_i$, $i\in \{1,2,3\}$, with $P_1$ having length at least 1 and $P_2$ and $P_3$ having length at least 2.
    \item A \textit{prism} consists of two triangles $a_1a_2a_3$ and $b_1b_2b_3$ and three disjoint nonadjacent induced paths $P_i$ from $a_i$ to $b_i$, $i\in \{1,2,3\}$, all of length at least 1.
    \item A \textit{wheel} consists of an induced cycle of length at least 4 plus a vertex with at least 3 neighbors in the cycle.
\end{itemize}

One instance of Problem~\ref{prob:general} is the following. The class of graphs with treewidth at most $k$ is given by forbidding a (finite) set of graph minors. Which induced subgraphs one must forbid in order to get a class of bounded treewidth? This question has, for instance, been investigated in the paper series starting with~\cite{tara}.

Other examples come from designing polynomial-time algorithms to detect given structures in graphs, particularly induced minors and Truemper configurations. Several key lemmas state that the presence of some induced minor forces the presence of certain induced subgraphs. For instance, in~\cite{DBLP:conf/iwoca/DallardDHMPT24}, it is proved that the presence of $K_{2, 3}$ as an induced minor is equivalent to the presence of some Truemper configurations. In~\cite{chktw}, it is proved that if $G$ contains $K_{3, 4}$ an induced minor, then it must contain a triangle or a theta. In contrast, it is proved in \cite{DBLP:journals/jgt/SintiariT21} that for all integers $t$, there exists a (theta, triangle)-free graph containing $K_t$ as an induced minor (or equivalently as a minor). This suggests the following specific instances of Problem~\ref{prob:general} are of interest:

\begin{problem}\label{prob:forcer}
Given a graph $H$, is there a (theta, triangle)-free graph $G$ that contains $H$ as an induced minor?
\end{problem}

\begin{problem}\label{prob:Fforcer}
Given a finite set of graphs $\Fc$ and a graph $H$, is there an $\Fc$-free graph $G$ that contains $H$ as an induced minor?
\end{problem}

Many other variations can be given. We do not know whether Problems~\ref{prob:forcer} or~\ref{prob:Fforcer} are decidable, or indeed if Problem~\ref{prob:general} is decidable (for some specific notions of graph structure). To investigate such questions, it would be useful to have a framework to express various kinds of containment relations and potential substructures in a common language. We propose here such a general framework that we call \textit{pathographs}.

Roughly speaking, a \textit{pathograph} $\Gf$ is a 6-tuple $(V, U, E, S, R, \pi)$, where:
\begin{itemize}
    \item $V$ is the set of vertices of $\Gf$;
    \item $U$ is the set of ``urpaths''\footnote{From ``ur-'', meaning primordial.} of $\Gf$, that are kind of edge to be subdivided;
    \item $E\subseteq \binom{V}{2}$ is the set of edges of $\Gf$, encoding which pairs of vertices are adjacent in $\Gf$;
    \item $S\subseteq V\times U$ is the set of ``spokes'' of $\Gf$, encoding which vertices must be adjacent to internal vertices of urpaths in realizations of $\Gf$;
    \item $R\subseteq \binom{U}{2}$ is the set of ``rungs'' of $\Gf$, encoding which pairs of urpaths have adjacent internal vertices in realizations of $\Gf$;
    \item $\pi:U\to \binom{V}{2}$ identifies the endpoints of urpaths.
\end{itemize}

A precise definition is given in Section~\ref{sec:pathograph}, and we also define the containment relation of pathographs and the notion of a \textit{realization} of a pathograph (roughly speaking, this is a graph formed by replacing all urpaths of a graph by induced paths, where the edges incident to internal vertices on these paths is restricted by the presence or absense of spokes and rungs incident to the corresponding urpath). We also check that pathographs can be used to encode all of the graph containment relations and particular structures we are interested in:

\begin{restatable}{theorem}{encode}\label{thm:encode}
Let $H$ be a graph. Then there are finite sets of pathographs
$\Sf_i(H)$, $i\in\{1,2,3,4,5,6\}$, such that:
\begin{enumerate}
    \item $G$ contains $H$ as a subgraph if and only if $G$ contains some $\Hf\in \Sf_1(H)$.
    \item $G$ contains $H$ as an induced subgraph if and only if $G$ contains some $\Hf\in \Sf_2(H)$.
    \item $G$ contains $H$ as a minor if and only if $G$ contains some $\Hf\in \Sf_3(H)$.
    \item $G$ contains $H$ as an induced minor if and only if $G$ contains some $\Hf\in \Sf_4(H)$.
    \item $G$ contains $H$ as a topological minor if and only if $G$ contains some $\Hf\in \Sf_5(H)$.
    \item $G$ contains $H$ as an induced topological minor if and only if $G$ contains some $\Hf\in \Sf_6(H)$.
\end{enumerate}

Additionally, there are finite sets of pathographs $\Theta$, $\Py$, $\Pr$, and $\W$ such that:
\begin{enumerate}[start=7]
    \item $G$ contains an induced theta if and only if $G$ contains some $\Hf\in\Theta$.
    \item $G$ contains an induced pyramid if and only if $G$ contains some $\Hf\in\Py$.
    \item $G$ contains an induced prism if and only if $G$ contains some $\Hf\in\Pr$.
    \item $G$ contains an induced wheel if and only if $G$ contains some $\Hf\in\W$.
\end{enumerate}
\end{restatable}

A pathograph is \textit{$\Fc$-free}, where $\Fc$ is a set of pathographs, if it does not contain any $\Ff\in\Fc$, under the pathograph containment relation defined in Section~\ref{sec:pathograph}. Now we introduce the main problem considered in this paper:

\begin{problem}\label{prob:realization}
Given a pathograph $\Hf$ and finite set of pathographs $\Fc$, is there an $\Fc$-free realization of $\Hf$?
\end{problem}

If this problem were decidable, this would give an algorithm to decide Problem~\ref{prob:general}, so long as the graph classes considered are defined using only the ten containment relations and structures listed in Theorem~\ref{thm:encode}. This is because there is an $\Fc$-free realization of $\Hf$ if and only if $S\not\subseteq S'$ where $S$ is the class of $\Fc$-free graphs and $S'$ is the class of $\Hf$-free graphs. In particular, it would give an algorithm to decide Problems~\ref{prob:forcer} and~\ref{prob:Fforcer}, and many others.

Alas, Problem~\ref{prob:realization} is not decidable in general. In Section~\ref{sec:undecidable}, we prove:

\begin{restatable}{theorem}{undecidable}\label{thm:undecidable}
Problem~\ref{prob:realization} is undecidable even if $\Hf$ has only one rung and all $\Ff\in\Fc$ have no urpaths, i.e.\ are conventional graphs.
\end{restatable}

The proof of Theorem~\ref{thm:undecidable} is by a reduction from the periodic Wang tiling problem. In particular, the pattern of edges between the two paths $P$ and $Q$ will encode the unit cell of a periodic tiling, and the set $\Fc$ will encode the Wang tiling rules, so that the tiling is valid if and only if the associated graph is $\Fc$-free. The idea of using Wang tiles was inspired by Braunfeld's proof that it is undecidable to determine if a hereditary graph class has the joint embedding property~\cite{braunfeld}.

This result can be restated using only conventional graph terminology as follows:
\begin{restatable}{corollary}{paths}\label{cor:paths}
The following problem is undecidable. Let $\Fc$ be a finite set of graphs, $H$ a graph with four chosen vertices $a,b,c,d$ and two chosen subsets of vertices $U_1$ and $U_2$. Does there exist an $\Fc$-free graph formed by adding to $H$ an induced path $P_1$ from $a$ to $b$ and a path $P_2$ from $c$ to $d$, so that the neighbors of the internal vertices of $P_i$ are restricted to lie in $P_1\cup P_2\cup U_i$?
\end{restatable}

Additionally, Problem~\ref{prob:realization} remains undecidable even if only $\Hf$ or $\Fc$ is allowed to vary, but not both. See Section~\ref{sec:variants} for a precise statement.

Based on these results, we believe the following is true, though we are unable to prove it:

\begin{conjecture}
The following problem is undecidable: let $\Sc$ and $\Sc'$ be families of graphs defined by forbidding finitely many (induced) subgraphs, (induced) minors, (induced) topological minors, and/or Truemper configurations; is $\Sc\subseteq \Sc'$?
\end{conjecture}

\subsection{Decidable versions}

We can, however, provide algorithms to decide Problem~\ref{prob:general} in some restricted cases. First, in Section~\ref{sec:rungless}, we prove:

\begin{restatable}{theorem}{rungless}\label{thm:rungless}
If $\Hf$ has no rungs, then Problem~\ref{prob:realization} is decidable.
\end{restatable}
We observe that pathographs with neither rungs nor spokes are very closely related to \textit{subdivisible graphs}, introduced by L\'ev\^eque, Lin, Maffray, and Trotignon to study the computational complexity of several induced subgraphs containment problems~\cite{llmt}.

Additionally, we prove:
\begin{restatable}{theorem}{lineartime}\label{thm:lineartime}
For any fixed pathograph $\Hf$ with no rungs and finite set of pathographs $\Fc$, there is a linear time algorithm to decide if a given realization of $\Hf$ is $\Fc$-free or not.
\end{restatable}

This is in stark contrast with the case of determining if $G$ is $\Fc$-free or not, which is NP-complete even when holding $\Fc$ fixed; this is implied, for example, by the fact that detecting induced prisms in graphs is NP-complete \cite{prismnpc}. Hence this theorem may be interpreted as a fixed-parameter tractability result. These two theorems are proved in Section~\ref{sec:rungless} by constructing a finite automaton that takes as input a realization $G$ of $\Hf$ and accepts if and only if $G$ is $\Fc$-free.

Theorem~\ref{thm:rungless} implies the following:
\begin{corollary}
Problem~\ref{prob:general} is decidable if $S'$ is given by forbidding finitely many induced subgraphs, induced topological minors, and Truemper configurations and $S$ is given by forbidding finitely many pathographs (in particular, any of the ten structures listed in Theorem~\ref{thm:encode}).
\end{corollary}

In fact, if there are $\Fc$-free realizations of $\Hf$ but they are of a restricted form, this is also automatically discoverable and provable; this is discussed further in Section~\ref{ssec:misc}. This situation occurs quite frequently in the proofs of ``decomposition theorems''. These are theorems of the general form ``If $G$ is in a class $S$ defined by forbidding some structures called \textit{obstructions}, it either lies in one of a few \textit{basic classes}, or else $G$ admits a \textit{decomposition}: it is built from smaller graphs in $S$ in a prescribed way''.

A powerful method to obtain decomposition theorems is to consider a basic graph $H$ (and therefore in the class $S$) that is in some sense ``close'' to containing an obstruction. Then, a general non-basic graph $G$ of the class is analyzed under the assumption that it contains $H$. The key is that since $H$ is ``close'' to containing an obstruction, the vertices in $G \setminus H$ ought to attach to $H$ in a sufficiently restricted way that entails a decomposition of $G$. So it remains to consider graphs in the class that do not contain $H$, which is to say we have a new problem of the same form but with one more obstruction, namely $H$. After several such steps, so many graphs are excluded that all remaining graphs under consideration are basic.

The oldest application of this paradigm seems to be the decomposition theorem of graphs containing no $K_5$ as a minor, due to Wagner~\cite{wagner:k5}. It is notably used by Seymour to decompose regular matroids~\cite{seymour:decRegMat} or by Chudnovsky, Robertson, Seymour and Thomas to prove the strong perfect graph theorem~\cite{chudnovsky.r.s.t:spgt}. All these famous examples are very involved and a reader not familiar with the method may find something much simpler in~\cite [Theorem~2.1]{CPST:bull}.

When applying this method, the proofs can be long and the verifications tedious. It was asked in~\cite[Question 2.12]{nicolas:hdr} if there is a general algorithm to prove such ``attachment lemmas''. Theorem~\ref{thm:rungless} provides such an algorithm in certain cases, while Theorem~\ref{thm:undecidable} shows that perhaps there is no algorithm that works in all cases.

Finally, we prove in Section~\ref{sec:closed}:

\begin{restatable}{theorem}{closeddecidable}\label{thm:closeddecidable}
Problem~\ref{prob:realization} is decidable if $\Fc$ is closed under adding edges, spokes, and rungs.
\end{restatable}

This implies:

\begin{corollary}
Problem~\ref{prob:general} is decidable if $S$ is given by forbidding finitely many finitely many forbidden subgraphs, minors, and/or topological minors (all non-induced), and $S'$ is given by forbidding finitely many pathographs (in particular, any of the ten structures listed in Theorem~\ref{thm:encode}).
\end{corollary}

Unfortunately, none of these special cases covers the graph classes considered in Problems~\ref{prob:forcer} and~\ref{prob:Fforcer}, where $S'$ is given by forbidding an induced minor and $S$ is given by forbidding induced subgraphs and Truemper configurations.

\section{Pathographs}
\label{sec:pathograph}

\subsection{Definitions}

In this section we give a precise definition of ``pathograph'' and related terms, and we enumerate some basic results about them.
As mentioned in the introduction, a \textit{pathograph} $\Gf$ if a 6-tuple $(V, U, E, S, R, \pi)$, where:

\begin{itemize}
    \item $V$ is the set of \textit{vertices} of $\Gf$. We assume $V$ is finite in this paper.
    \item $U$ is the set of \textit{urpaths} of $\Gf$. We assume $U$ is finite in this paper.
    \item $E\subseteq \binom{V}{2}$ is the set of \textit{edges} of $\Gf$. If $\{v_1,v_2\}\in E$ we say $v_1$ and $v_2$ are \textit{adjacent}; otherwise we say they are \textit{nonadjacent}. We say $v_1$ and $v_2$ are \textit{incident} to the edge $\{v_1,v_2\}$.
    \item $S\subseteq V\times U$ is the set of \textit{spokes} of $\Gf$. If $(v, u)\in S$, where $v$ is a vertex and $u$ is an urpath, we say $v$ and $u$ are \textit{adjacent}; otherwise \textit{nonadjacent}. We say $v$ and $u$ are \textit{incident} to the spoke $(v, u)$.
    \item $R\subseteq \binom{U}{2}$ is the set of \textit{rungs} of $\Gf$. If $\{u_1,u_2\}\in R$ we say $u_1$ and $u_2$ are \textit{adjacent}; otherwise \textit{nonadjacent}. We say $u_1$ and $u_2$ are \textit{incident} to the rung $(u_1, u_2)$.
    \item $\pi:U\to \binom{V}{2}$ sends each urpath to its \textit{endpoints}. Note that the endpoints of an urpath must be distinct vertices. If $v\in \pi(u)$ we say $v$ is \textit{incident} to $u$. We require the endpoints of each urpath to be nonadjacent, i.e.\ $\im(\pi)\cap E=\varnothing$.
\end{itemize}
Usually we suppress $\pi$ and just write, for example, ``the endpoints of urpath $u$ are $(v_1,v_2)$'' to mean $\pi(u)=\{v_1,v_2\}$. It will often be helpful to view the urpaths as having a ``left'' and ``right'' endpoint; which endpoint is which can be chosen arbitrarily, but in the expression $(v_1,v_2)$ we are identifying $v_1$ as the left endpoint and $v_2$ as the right endpoint. 

A \textit{realization} of $\Gf$ is a graph $G$ formed by replacing each urpath $u_i$ in $\Gf$, say with endpoints $(a_i,b_i)$, with an induced path $P_i=a_ix_i^{(1)}x_i^{(2)}\cdots x_i^{(m_i)}b_i$, with $m_i\ge 1$, such that:
\begin{itemize}
    \item If vertex $v_j$ is adjacent to $u_i$ (an urpath) in $\Gf$, then $v_j$ and $x_i^{(k)}$ are adjacent in $G$ for some $k$. If $v_j$ and $u_i$ are nonadjacent, then $v_j$ is not adjacent to any $x_i^{(k)}$.
    \item If urpaths $u_i$ and $u_j$ are adjacent in $\Gf$, then $x_i^{(k)}$ and $x_j^{(\ell)}$ are adjacent in $G$ for some $k,\ell$. If $u_i$ and $u_j$ are nonadjacent, then $x_i^{(k)}$ and $x_j^{(\ell)}$ are nonadjacent for all $k,\ell$.
\end{itemize}
The vertices and edges of $\Gf$ are preserved in $G$. When drawing pathographs, urpaths are drawn as squiggly lines. Spokes $(v,u)$ are drawn using a line that looks like \begin{tikzpicture}
\draw[-|] (0,0) -- (0.5,0);
\end{tikzpicture} with the ``flat end'' of the line pointing to the urpath $u$ and the other end attached to the vertex $v$. Rungs are drawn using a line that looks like \begin{tikzpicture}
\draw[|-|] (0,0) -- (0.5,0);
\end{tikzpicture}. We give an example pathograph and realization in Figure~\ref{fig:expathograph}.

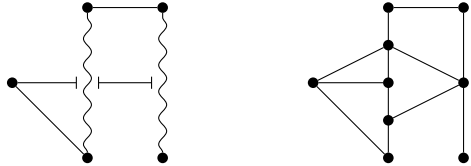
\begin{figure}[htbp!]
    \centering
    \begin{tikzpicture}
    \vert{0}{0,0};
    \vert{1}{1,-1};
    \vert{2}{1,1};
    \vert{3}{2,-1};
    \vert{4}{2,1};
    \edge{0}{1};
    \edge{2}{4};
    \urpath{1}{2};
    \urpath{3}{4};
    \spoke{0}{1}{2};
    \rung{1}{2}{3}{4};

    \vert{0}{4,0};
    \vert{1}{5,-1};
    \vert{2}{5,1};
    \vert{3}{6,-1};
    \vert{4}{6,1};
    \vert{5}{5,-0.5};
    \vert{6}{5,0};
    \vert{7}{5,0.5};
    \vert{8}{6,0};
    \edge{0}{1};
    \edge{2}{4};
    \edge{1}{2};
    \edge{3}{4};
    \edge{0}{6};
    \edge{0}{7};
    \edge{5}{8};
    \edge{7}{8};
    \end{tikzpicture}
    \caption{An example pathograph $\Gf$ (left) and realization $G$ (right). Here, $\Gf$ has 5 vertices, 2 urpaths, 2 edges, 1 spoke, and 1 rung.}
    \label{fig:expathograph}
\end{figure}

Note that all conventional graphs may be viewed as pathographs with no urpaths, spokes, or rungs; we convert freely between the two in this case. We write $V(\Gf),U(\Gf)$, etc.\ to denote the vertices, urpaths, etc.\ of $\Gf$. A \textit{subpathograph} of $\Gf$ is a pathograph obtained by deleting some vertices and/or urpaths from $\Gf$. This should be thought of as generalizing induced subgraphs of graphs. When a vertex is deleted from $\Gf$, we also delete all urpaths, edges, and spokes that vertex was incident to; when an urpath is deleted from $\Gf$, we also delete all spokes and rungs that urpath was incident to (but we need not delete its endpoints).

If the vertices and urpaths of $\Gf$ can be partitioned into two nonempty (disjoint) subpathographs that have no edges, spokes, or rungs between them, then $\Gf$ is \textit{disconnected}; otherwise, $\Gf$ is \textit{connected}. Additionally, the pathograph with no vertices is defined to be disconnected. Note that a pathograph is connected if and only if all of its realizations are connected (following the convention that the empty graph is disconnected). A \textit{path} in $\Gf$ is one of the following:
\begin{enumerate}
    \item a connected subpathograph with no spokes or rungs, such that every vertex is incident to exactly 2 edges or urpaths with two exceptions, and those two exceptional vertices are each incident to exactly 1 edge or urpath; or
    \item a single vertex of $\Gf$.
\end{enumerate}
We write $P(\Gf)$, a set of subpathographs of $\Gf$, to denote the paths of $\Gf$. Note that if $G$ is a conventional graph, i.e.\ a pathograph with no urpaths, $P(G)$ is precisely the set of induced paths in $G$ (allowing for 1-vertex paths). A vertex $v$ on a path $P$ is an \textit{endpoint} if it is incident to at most one urpath and edge in $P$, and $v$ is an \textit{internal vertex} otherwise. We say two paths $P_1$ and $P_2$ are \textit{adjacent} if a vertex or urpath on $P_1$ is adjacent to a vertex or urpath on $P_2$, i.e.\ if there are any edges, spokes, or rungs between $P_1$ and $P_2$.

We say that a graph $G$ \textit{contains} $\Gf$ if some induced subgraph of $G$ is isomorphic to a realization of $\Gf$. More generally, we can define the notation of a \textit{pathograph inclusion} $\phi:\Hf\to \Gf$. This is a pair $\phi = (\phi_V, \phi_U)$ where:

\begin{itemize}
    \item $\phi_V:V(\Hf)\to V(\Gf)$ is injective.
    \item $\phi_U:U(\Hf)\to P(\Gf)$ has the following properties:
    \begin{itemize}
        \item Let $u$ be an urpath in $\Hf$ with endpoints $(a,b)$. Then the endpoints of $\phi_U(u)$ are $\phi_V(a)$ and $\phi_V(b)$ (note that there must be two endpoints of the path $\phi_U(u)$ due to the requirements that $\phi_V$ is injective and the endpoints of $u$ are different).
        \item $\im(\phi_U)$ is a set of internally vertex-disjoint (but not necessarily nonadjacent) paths in $\Gf$, each of which contains either at least one urpath or at least three vertices.
    \end{itemize}
    \item For any $a,b\in V(\Gf)\cup U(\Gf)$, $a$ and $b$ are adjacent in $\Hf$ if and only if $\phi(a)$ and $\phi(b)$ are adjacent in $\Gf$, where by $\phi(x)$ we mean either $\phi_V(x)$ or $\phi_U(x)$ depending on whether $x$ is a vertex or urpath.
\end{itemize}
Say that $\Gf$ and $\Hf$ are \textit{isomorphic} if there are pathograph inclusions $\Gf\to\Hf$ and $\Hf\to\Gf$.

We remark that pathograph inclusions may be composed in an obvious way. If $\phi:\Gf_1\to \Gf_2$ and $\psi:\Gf_2\to \Gf_3$ are pathograph inclusions, then the composition $\xi:\psi\circ\phi$ given by $\xi_V=\psi_V\circ\phi_V$ and $\xi_U(u)=\{\psi(x)\mid x\in \phi(u)\}$ is a pathograph inclusion $\Gf_1\to \Gf_3$.

One can easily check all of the following basic properties:
\begin{proposition}\label{prop:basicfacts}
The following facts hold:
\begin{itemize}
    \item A graph $G$ contains $\Hf$ (i.e.\ contains a realization of $\Hf$ as an induced subgraph) if and only if there is a pathograph inclusion $\Hf\to G$.
    \item A graph $G$ is a realization of $\Hf$ if and only if there is a pathograph inclusion $\phi:\Hf\to G$ such that each vertex of $G$ either lies in the image of $\phi_V$ or lies on some path in the image of $\phi_U$.
    \item Say $\Hf \subseteq \Gf$ if there is a pathograph inclusion $\Hf\to\Gf$. Then $\subseteq$ is a partial order on pathographs up to isomorphism (i.e.\ it is reflexive, antisymmetric, and transitive).
\end{itemize}\qed
\end{proposition}

We also briefly mention that pathographs can be further generalized by marking urpaths as ``even'' or ``odd'', and in realizations they must be replaced by even- or odd-length paths, respectively. This allows one to study, for example, the relationship between even hole-free graphs and other graph classes.

\subsection{Encoding other graph containment relations}

Let us prove the following:
\encode*
\begin{proof}
For a set of pathographs $S$, let $\cl_\sim(S)$ be the set of pathographs formed from some $\Hf\in S$ by adding any number of edges, spokes, or rungs (including none). Also, let $\cl_U(S)$ be the set of pathographs formed from some $\Hf\in S$ by replacing any number of edges with urpaths incident to the same vertices (i.e.\ $e=\{v_1,v_2\}$ may be replaced by $u$ with $\pi(u)=\{v_1,v_2\}$). This allows us to deal with (induced) subgraphs and (induced) topological minors:
\begin{enumerate}
    \item[1, 2, 5, 6.] Obviously we may take $\Sf_1(H)=\cl_\sim(\{H\})$, $\Sf_2(H)=\{H\}$, $\Sf_5(H)=\cl_\sim(\cl_U(\{H\}))$, and $\Sf_6(H)=\cl_U(\{H\})$.
\end{enumerate}

Minors and induced minors are slightly more complicated. For a nonnegative integer $k$, let $\conn(k)$ be the set of connected pathographs $\Hf$ with at most $k$ vertices (and at least one vertex, since the empty pathograph is disconnected), such that removing any urpath from $\Hf$ leaves a disconnected pathograph. Note that $\conn(k)$ is a finite set for all $k$. Let $\cl_M(H)$ be the set of pathographs $\Hf$ such that all of the following hold:
\begin{itemize}
    \item $\Hf$ may be partitioned into $\abs{V(H)}$ connected pathographs $C_1,\dots,C_{\abs{V(H)}}$, such that each $C_i$ is in $\conn(\max(1,\deg(v_i)))$, where $v_1,\dots,v_{\abs{V(H)}}$ are the vertices of $H$.
    \item There is at least one edge, spoke, or rung from $C_i$ to $C_j$ if and only if $v_i$ and $v_j$ are adjacent in $H$.
\end{itemize}
Then:
\begin{enumerate}
    \item[4.] We claim that we may take $\Sf_4(H)=\cl_M(H)$. Indeed, suppose $G$ contains some $\Hf\in \cl_M(H)$. Then $G$ contains some disjoint connected induced subgraphs $C_1',\dots,C_{\abs{V(H)}}'$, such that some vertex of $C_i'$ is adjacent to some vertex of $C_j'$ if and only if $v_i$ is adjacent to $v_j$, where $v_1,\dots,v_{\abs{V(H)}}$ are the vertices of $H$. This means $G$ contains $H$ as an induced minor. Conversely, if $G$ contains $H$ as an induced minor, then it contains some disjoint connected induced subgraphs $C_1',\dots,C_{\abs{V(H)}}'$, such that $C_i'$ is adjacent to $C_j'$ if and only if $v_i$ is adjacent to $v_j$. For each $i,j$ with $v_i$ adjacent to $v_j$, choose a vertex $x_{i,j}\in C_i'$ adjacent to some vertex in $C_j'$. If $v_i$ has degree 0, then choose an arbitrary vertex $x$ in $C_i'$. Then it is easy to see that $C_i'$ contains some $\Cf\in \conn(\max(1,\deg(v_i)))$ (the choice of $x_{i,j}$ and/or $x$ determines the image of $\phi_V:V(\Cf)\to C_i'$). From this we can deduce that $G$ contains some $\Hf\in \cl_M(H)$.
    \item[3.] Given the above discussion, it is obvious we may take $\Sf_3(H)=\cl_\sim(\cl_M(H))$.
\end{enumerate}

Finally, consider the pathographs in Figure~\ref{fig:truemper}.
\begin{figure}[htbp!]
    \centering
    \begin{tabular}{c c c c c}
        \begin{tikzpicture}
        \useasboundingbox (0,0) -- (2,2);
        \vert{1}{0,1};
        \vert{2}{2,1};
        \urpath{1}{2};
        \draw[decorate,decoration={snake,amplitude=1.5,segment length=10, post length=0,pre length=0}] (1) to[bend right=90] (2);
        \draw[decorate,decoration={snake,amplitude=1.5,segment length=10, post length=0,pre length=0}] (1) to[bend left=90] (2);
        \end{tikzpicture} &
        \begin{tikzpicture}
        \useasboundingbox (0,0) -- (2,2);
        \vert{1}{0,0};
        \vert{2}{0,2};
        \vert{3}{0.5,1};
        \vert{4}{1.5,1};
        \vert{5}{2,2};
        \vert{6}{2,0};
        \edge{1}{2};
        \edge{1}{3};
        \edge{2}{3};
        \edge{4}{5};
        \edge{4}{6};
        \edge{5}{6};
        \edge{3}{4};
        \edge{2}{5};
        \edge{1}{6};
        \end{tikzpicture} &
        \begin{tikzpicture}
        \useasboundingbox (0,0) -- (2,2);
        \vert{1}{0,0};
        \vert{2}{0,2};
        \vert{3}{0.5,1};
        \vert{4}{1.5,1};
        \vert{5}{2,2};
        \vert{6}{2,0};
        \edge{1}{2};
        \edge{1}{3};
        \edge{2}{3};
        \edge{4}{5};
        \edge{4}{6};
        \edge{5}{6};
        \edge{3}{4};
        \urpath{2}{5};
        \edge{1}{6};
        \end{tikzpicture} &
        \begin{tikzpicture}
        \useasboundingbox (0,0) -- (2,2);
        \vert{1}{0,0};
        \vert{2}{0,2};
        \vert{3}{0.5,1};
        \vert{4}{1.5,1};
        \vert{5}{2,2};
        \vert{6}{2,0};
        \edge{1}{2};
        \edge{1}{3};
        \edge{2}{3};
        \edge{4}{5};
        \edge{4}{6};
        \edge{5}{6};
        \urpath{3}{4};
        \urpath{2}{5};
        \edge{1}{6};
        \end{tikzpicture} &
        \begin{tikzpicture}
        \useasboundingbox (0,0) -- (2,2);
        \vert{1}{0,0};
        \vert{2}{0,2};
        \vert{3}{0.5,1};
        \vert{4}{1.5,1};
        \vert{5}{2,2};
        \vert{6}{2,0};
        \edge{1}{2};
        \edge{1}{3};
        \edge{2}{3};
        \edge{4}{5};
        \edge{4}{6};
        \edge{5}{6};
        \urpath{3}{4};
        \urpath{2}{5};
        \urpath{1}{6};
        \end{tikzpicture} \\
        $\Theta_1$ & $\Pr_1$ & $\Pr_2$ & $\Pr_3$ & $\Pr_4$ \\
        \\
        \begin{tikzpicture}
        \useasboundingbox (0,0) -- (2,2);
        \vert{1}{0,0};
        \vert{2}{0,2};
        \vert{3}{0.5,1};
        \vert{4}{2,1};
        \edge{1}{2};
        \edge{1}{3};
        \edge{2}{3};
        \urpath{1}{4};
        \urpath{2}{4};
        \edge{3}{4};
        \end{tikzpicture} &
        \begin{tikzpicture}
        \useasboundingbox (0,0) -- (2,2);
        \vert{1}{0,0};
        \vert{2}{0,2};
        \vert{3}{0.5,1};
        \vert{4}{2,1};
        \edge{1}{2};
        \edge{1}{3};
        \edge{2}{3};
        \urpath{1}{4};
        \urpath{2}{4};
        \urpath{3}{4};
        \end{tikzpicture} &
        \begin{tikzpicture}
        \useasboundingbox (0,0) -- (2,2);
        \vert{1}{0,1};
        \vert{2}{1,1};
        \vert{3}{2,1};
        \draw[decorate,decoration={snake,amplitude=1.5,segment length=10, post length=0,pre length=0}] (1) to[bend right=90] (3);
        \draw[decorate,decoration={snake,amplitude=1.5,segment length=10, post length=0,pre length=0}] (1) to[bend left=90] (3);
        \edge{1}{2};
        \edge{2}{3};
        \manspoke{2}{1,0.25};
        \end{tikzpicture} &
        \begin{tikzpicture}
        \useasboundingbox (0,0) -- (2,2);
        \vert{1}{0,1};
        \vert{2}{1,1};
        \vert{3}{2,1};
        \draw[decorate,decoration={snake,amplitude=1.5,segment length=10, post length=0,pre length=0}] (1) to[bend right=90] (3);
        \draw[decorate,decoration={snake,amplitude=1.5,segment length=10, post length=0,pre length=0}] (1) to[bend left=90] (3);
        \edge{1}{2};
        \edge{2}{3};
        \manspoke{2}{1,0.25};
        \manspoke{2}{1,1.75};
        \end{tikzpicture} \\
        $\Py_1$ & $\Py_2$ & $\W_1$ & $\W_2$
    \end{tabular}
    \caption{Pathographs corresponding to Truemper configurations.}
    \label{fig:truemper}
\end{figure}
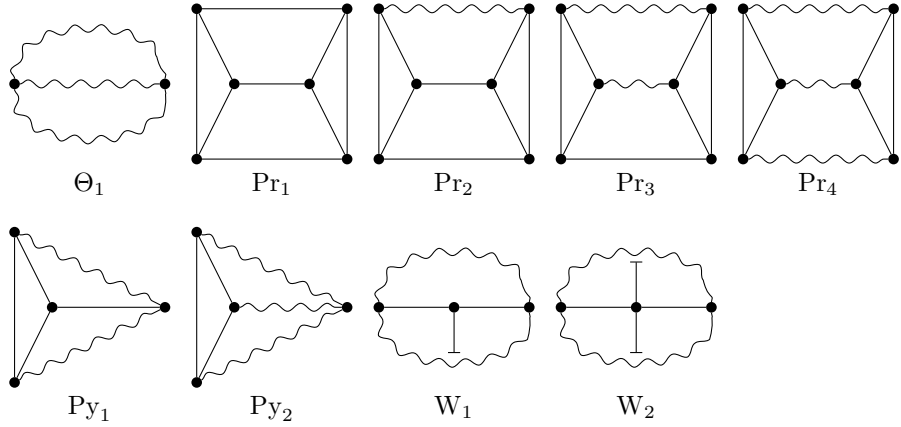
\begin{enumerate}
    \item[7, 8, 9, 10.] Obviously we may take $\Theta=\{\Theta_1\}$, $\Py=\{\Py_1,\Py_2\}$, $\Pr=\{\Pr_1,\Pr_2,\Pr_3,\Pr_4\}$, and $\W=\{\W_1,\W_2\}$.
\end{enumerate}
\end{proof}

\section{The general pathograph realization problem}
\label{sec:undecidable}

The main goal of this section is to prove the following:

\undecidable*

To prove this theorem, we reduce from the problem of determining if a set of Wang tiles admits a periodic tiling. Recall that a \textit{Wang tile} is a 4-tuple $(N, E, S, W)\in C^4$ for some base set $C$ of colors. A set of Wang tiles $S$ is said to \textit{tile the plane} if there is a function $f:\Z^2\to S$ such that for all $i,j$, $f(i,j)_{(2)}=f(i+1,j)_{(4)}$ and $f(i,j)_{(1)}=f(i,j+1)_{(3)}$, where by $X_{(k)}$ we mean the $k$th element of $X$; in other words, we must cover the plane by translates of the tiles in $S$ so that the colors of adjacent tiles match up everywhere. A set of Wang tiles is said to \textit{periodically tile the plane} if there is such a tiling $f$ and positive integers $a,b$ such that $f(i,j)=f(i+a,j)=f(i,j+b)$ for all integers $i,j$; then $f$ is said to be \textit{$(a,b)$-periodic}. We now recall the following classic undecidability results:

\begin{theorem*}\label{thm:wangundecidable}
The following problems are undecidable:
\begin{enumerate}
    \item Given a finite set of Wang tiles $S$, does $S$ tile the plane? \cite{berger}
    \item Given a finite set of Wang tiles $S$, does $S$ periodically tile the plane? \cite{gk}
\end{enumerate}
\end{theorem*}

We will need a variant of the second problem:

\begin{corollary}\label{cor:specialundecidable}
The following problem is undecidable for any positive integer $n$: given a finite set of Wang tiles $S$, a function $f':[1,n]^2\cap \Z^2 \to S$, is there a periodic tiling $f$ of $S$ such that $f|_{[1,n]^2\cap \Z^2}=f'$?
\end{corollary}
\begin{proof}
If this was decidable, one could construct an algorithm for the second problem in the theorem above by simply testing for each of the (finitely many) functions $f':[1,n]^2\cap \Z^2 \to S$ whether there is a periodic tiling $f$ with $f|_{[1,n]^2\cap \Z^2}=f'$.
\end{proof}

We first prove two variations of Theorem~\ref{thm:undecidable}. The first variation allows for the vertices, edges, and urpaths of the relevant objects to be colored, and the edges, urpaths, spokes, and rungs to be directed. This makes encoding the Wang tile rules using forbidden pathographs simpler. The second variation restricts the coloring to just vertices and removes the directedness of edges, urpaths, spokes, and rungs; this is accomplished by replacing vertices by short paths of cleverly colored vertices and replacing edges by special bipartite graphs depending on what color they used to be, and adding some forbidden pathographs to enforce this additional structure. The final proof will remove the colors from vertices and leave us with conventional pathographs; this is accomplished by the addition of a special gadget and attaching all of the vertices into this gadget according to the color they used to be, and adding yet more forbidden pathographs to enforce this additional structure.

\subsection{Directed multicolored pathographs}

For the first variation, generalize pathographs to \textit{directed multicolored pathographs}: now, the vertices, edges, and urpaths are colored from a base set $D$, and edges, urpaths, spokes, and rungs are directed (but spokes and rungs are not colored). A \textit{realization} $G$ of such an object $\Gf$ is a directed multicolored graph (i.e.\ a directed graph with colored vertices and edges) where:
\begin{itemize}
    \item each directed urpath $u_i$ of $\Gf$ is replaced by a directed path $P_i$ of $G$ or length at least 3 such that all of the vertices and edges in $P_i$ are the color of the urpath,
    \item for each (directed) spoke $(v_i,u_j)$, there is at least one edge (directed in the same direction as the spoke, and of any color) between the internal vertices of $P_j$ and vertex $v_i$, and
    \item for each (directed) rung $(u_i,u_j)$, there is at least one edge from an internal vertex of $P_i$ to an internal vertex of $P_j$.
\end{itemize}
An \textit{induced subgraph} of a directed multicolored graph $G$ consists of some vertices $V'\subseteq V(G)$ and all the edges of $G$ between them; in other words, it is formed by vertex deletion in $G$. Two directed multicolored graphs are \textit{isomorphic} if there is a bijection between their vertices that preserves the colors of all vertices and direction and color of edges.

For our purposes, we forbid directed cycles of length 2. In other words, between any two vertices, there is either no edge, an edge from the first vertex to the second, or an edge from the second vertex to the first.

\begin{lemma}\label{lem:dicolor}
The following problem is undecidable: let $\Hf$ be a directed multicolored pathograph and $\Fc$ a finite set of directed multicolored pathographs; is there an $\Fc$-free realization of $\Hf$? In fact, it is undecidable even if $\Hf$ has just one rung and all the elements of $\Fc$ have no urpaths; i.e.\ they are conventional directed multicolored graphs.
\end{lemma}
Here, as before, ``$G$ is $\Fc$-free'' means that no induced subgraph of $G$ is isomorphic to a realization of any $\Ff\in \Fc$.
\begin{proof}
We reduce from the $n=3$ case of the problem considered in Corollary~\ref{cor:specialundecidable}. Let $S$ be a finite set of Wang tiles and $f':[1,3]^2\cap \Z^2 \to S$ a function. Our set $D$ of colors appearing in $\Hf$ and the pathographs of $\Fc$ will be $S$, the set of Wang tiles, plus two special colors, which we will call ``blue'' and ``red.'' The pathograph $\Hf$ consists of two directed cycles, one with red vertices and one with blue vertices, each with three vertices and one urpath, and all possible edges, spokes, and rungs from the red cycle to the blue cycle. Call the red vertices $x_1,x_2,x_3$ and the blue vertices $y_1,y_2,y_3$ such that the urpaths are from vertices $x_3$ to $x_1$ and $y_3$ to $y_1$. Finally, the color of the directed edge $x_iy_j$ is defined to be $f'(x_i,y_j)$ for each $i,j\in\{1,2,3\}$. The pathograph $\Hf$ is shown in Figure~\ref{fig:dimultiH}.

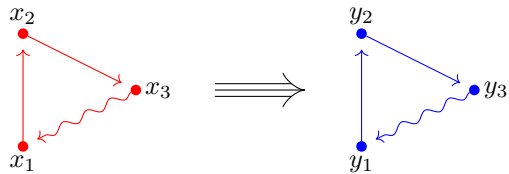
\begin{figure}[htbp!]
    \centering
    \begin{tikzpicture}[scale=1.5]
    \aclabvert{1}{0,0}{x_1}{below}{red};
    \aclabvert{2}{0,1}{x_2}{above}{red};
    \aclabvert{3}{1,0.5}{x_3}{right}{red};
    \aclabvert{4}{3,0}{y_1}{below}{blue};
    \aclabvert{5}{3,1}{y_2}{above}{blue};
    \aclabvert{6}{4,0.5}{y_3}{right}{blue};

    \dicoledge{1}{2}{red};
    \dicoledge{2}{3}{red};
    \dicolurpath{3}{1}{red};
    \dicoledge{4}{5}{blue};
    \dicoledge{5}{6}{blue};
    \dicolurpath{6}{4}{blue};

    \draw[double distance=2pt, arrows={-Implies}, scaling nfold=3] (1.7,0.5) -- (2.5, 0.5);
    \end{tikzpicture}
    \caption{The directed multicolored pathograph $\Hf$ in the case $n=3$. The triple arrow indicates that every possible (directed) edge, spoke, and rung is drawn from the left to the right. There will be a total of 9 edges, 3 spokes, and 1 rung between the sides. The color of the directed edge $x_iy_j$ is $f'(i,j)\in S$.}
    \label{fig:dimultiH}
\end{figure}

Now, $\Fc$ will contain the following, which are all pathographs with no urpaths:
\begin{enumerate}[start=2]
    \item[1a.] A pathograph with one red vertex $r$ and one blue vertex $b$, and a red edge from $r$ to $b$.
    \item[1b.] A pathograph with one red vertex $r$ and one blue vertex $b$, and a blue edge from $r$ to $b$.
    \item[1c.] A pathograph with one red vertex $r$ and one blue vertex $b$, and no edges.
    \item[1d.] For each color, a pathograph with one red vertex and one blue vertex and an edge of that color from the blue vertex to the red vertex.
    \item For each pair of tiles $s,t\in S$ such that $s_{(2)}\ne t_{(4)}$ (so tile $s$ cannot appear immediately to the left of $t$ in a valid tiling), a pathograph with two red vertices $r_1,r_2$ and one blue vertex $b$, with a red edge from $r_1$ to $r_2$, an edge of color $s$ from $r_1$ to $b$, and an edge of color $t$ from $r_2$ to $b$.
    \item For each pair of tiles $s,t\in S$ such that $s_{(1)}\ne t_{(3)}$ (so tile $s$ cannot appear immediately below $t$ in a valid tiling), a pathograph with two blue vertices $b_1,b_2$ and one red vertex $r$, with a blue edge from $b_1$ to $b_2$, an edge of color $s$ from $r$ to $b_1$, and an edge of color $t$ from $r$ to $b_2$.
\end{enumerate}
See Figure~\ref{fig:dicolorF} for a small example. We will refer to the pathographs in $\Fc$ by their numbers above; for instance, the pathograph defined in first entry will be called ``type 1a''. The first three graphs will collectively be referred to as ``type 1''.

\begin{figure}[htbp!]
    \centering
    \begin{tikzpicture}
    \draw[green, ultra thick] (0,0) -- (1,0);
    \draw[green, ultra thick] (1,0) -- (1,1);
    \draw[magenta, ultra thick] (1,1) -- (0,1);
    \draw[magenta, ultra thick] (0,1) -- (0,0);
    \node at (0.5, 0.5) {$s$};
    
    \draw[green, ultra thick] (2,0) -- (3,0);
    \draw[magenta, ultra thick] (3,0) -- (3,1);
    \draw[green, ultra thick] (3,1) -- (2,1);
    \draw[green, ultra thick] (2,1) -- (2,0);
    \node at (2.5, 0.5) {$t$};

    \tikzset{shift={(-2,-1.5)}};
    \node at (-0.5,0) {1a.};
    \cvert{0}{0,0}{red};
    \cvert{1}{1,0}{blue};
    \dicoledge{0}{1}{red};

    \tikzset{shift={(3,0)}};
    \node at (-0.5,0) {1b.};
    \cvert{0}{0,0}{red};
    \cvert{1}{1,0}{blue};
    \dicoledge{0}{1}{blue};

    \tikzset{shift={(3,0)}};
    \node at (-0.5,0) {1c.};
    \cvert{0}{0,0}{red};
    \cvert{1}{1,0}{blue};

    \tikzset{shift={(-7.5,-1.5)}};
    \node at (-0.5,0) {1d.};
    \cvert{0}{0,0}{red};
    \cvert{1}{1,0}{blue};
    \dicoledge{1}{0}{red};
    \node at (1.5,0) {$\vphantom{4},$};

    \tikzset{shift={(3,0)}};
    \cvert{0}{0,0}{red};
    \cvert{1}{1,0}{blue};
    \dicoledge{1}{0}{blue};
    \node at (1.5,0) {$\vphantom{4},$};

    \tikzset{shift={(3,0)}};
    \cvert{0}{0,0}{red};
    \cvert{1}{1,0}{blue};
    \dilabedge{1}{0}{s}{above};
    \node at (1.5,0) {$\vphantom{4},$};
    \tikzset{shift={(3,0)}};
    \cvert{0}{0,0}{red};
    \cvert{1}{1,0}{blue};
    \dilabedge{1}{0}{t}{above};
    \node at (1.5,0) {$\vphantom{4},$};

    \tikzset{shift={(-7.5,-1.5)}};
    \node at (-0.5,0) {2.};
    \cvert{0}{0,0.5}{red};
    \cvert{1}{0,-0.5}{red};
    \cvert{2}{1,0}{blue};
    \dicoledge{0}{1}{red};
    \dilabedge{0}{2}{s}{above};
    \dilabedge{1}{2}{s}{below};
    \node at (1.5,0) {$\vphantom{4},$};

    \tikzset{shift={(3,0)}};
    \cvert{0}{0,0.5}{red};
    \cvert{1}{0,-0.5}{red};
    \cvert{2}{1,0}{blue};
    \dicoledge{0}{1}{red};
    \dilabedge{0}{2}{t}{above};
    \dilabedge{1}{2}{s}{below};
    \node at (1.5,0) {$\vphantom{4},$};

    \tikzset{shift={(3,0)}};
    \cvert{0}{0,0.5}{red};
    \cvert{1}{0,-0.5}{red};
    \cvert{2}{1,0}{blue};
    \dicoledge{0}{1}{red};
    \dilabedge{0}{2}{t}{above};
    \dilabedge{1}{2}{t}{below};

    \tikzset{shift={(-5,-2)}};
    \node at (-0.5,0) {3.};
    \cvert{0}{0,0}{red};
    \cvert{1}{1,0.5}{blue};
    \cvert{2}{1,-0.5}{blue};
    \dicoledge{1}{2}{blue};
    \dilabedge{0}{1}{s}{above};
    \dilabedge{0}{2}{s}{below};
    \node at (1.5,0) {$\vphantom{4},$};

    \tikzset{shift={(3,0)}};
    \cvert{0}{0,0}{red};
    \cvert{1}{1,0.5}{blue};
    \cvert{2}{1,-0.5}{blue};
    \dicoledge{1}{2}{blue};
    \dilabedge{0}{1}{s}{above};
    \dilabedge{0}{2}{t}{below};
    \end{tikzpicture}
    \caption{An example $\Fc$. Here, $S$ consists of the two Wang tiles $s,t$ shown at the top of the figure, using colors $C=\{\text{green}, \text{magenta}\}$. Labels of black edges indicate their color.}
    \label{fig:dicolorF}
\end{figure}
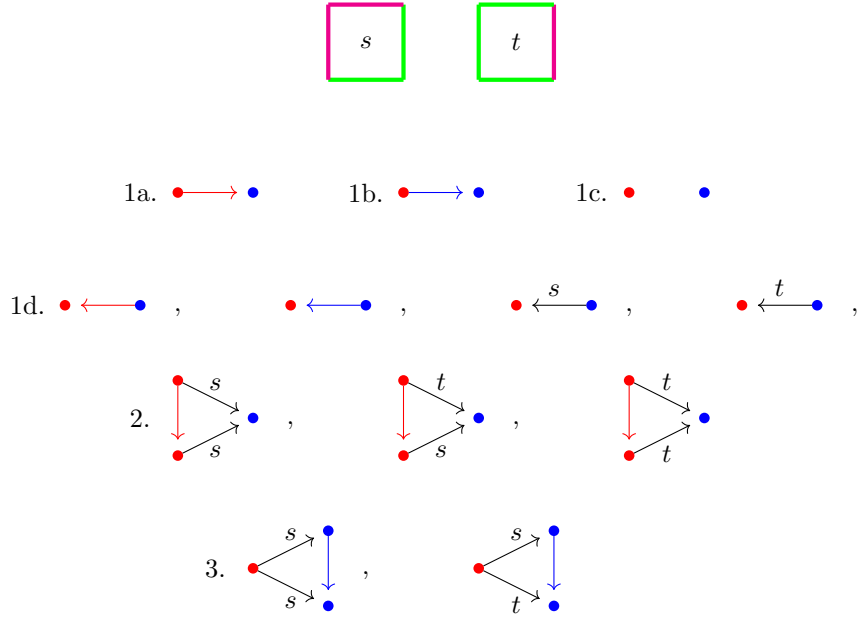

We claim that $S$ admits a periodic tiling $f$ with $f|_{[1,3]^2\cap \Z^2}=f'$ if and only if there is an $\Fc$-free realization of $\Hf$.

($\Rightarrow$) Suppose there is such a periodic tiling $f$, say $f$ is $(a,b)$-periodic, with $a,b\ge 4$ (this restriction on $a$ and $b$ causes no problems, since $(a,b)$-periodic tilings are also $(4a,4b)$-periodic). Then let $G$ be a directed multicolored graph with $a$ red vertices $x_1,\dots,x_a$, $b$ blue vertices $y_1,\dots,y_b$, and the following edges:
\begin{itemize}
    \item A red edge from $x_i$ to $x_{i+1}$ for each $i$, taking indices modulo $a$.
    \item A blue edge from $y_i$ to $y_{i+1}$ for each $i$, taking indices modulo $b$.
    \item An edge of color $f(i,j)$ from $x_i$ to $y_j$ for each $i,j$.
\end{itemize}
It is easy to verify this is indeed an $\Fc$-free realization of $\Hf$. It contains no graphs of type 1 by construction, and it has no graphs of type 2 or 3 since $f$ is a valid tiling.

($\Leftarrow$) Suppose there is such a realization $G$. Call the vertices on the red urpath $x_3,x_4,\dots,x_a,x_1$ and the vertices on the blue urpath $y_3,y_4,\dots,y_b,y_1$. First note that, by the pathographs of type 1 appearing in $\Fc$, there must be an edge with a color in $S$ (i.e.\ not red or blue) from $x_i$ to $y_j$ for all $i,j$. The type 1c pathograph ensures there is \textit{an} edge, and the types 1a and 1b pathographs make sure this edge is not red or blue. Then a periodic tiling $f$ with the desired properties is given by $f(i,j)\coloneq\text{the color of the edge from $x_i$ to $y_j$}$, taking the index of $x_i$ modulo $a$ and the index of $y_j$ modulo $b$. Note for each $i,j$ that:
\begin{itemize}
    \item by the type 2 pathographs appearing in $\Fc$, the east-facing color of tile $f(i,j)$ matches the west-facing color of tile $f(i+1,j)$, and
    \item by the type 3 pathographs appearing in $\Fc$, the north-facing color of tile $f(i,j)$ matches the south-facing color of tile $f(i,j+1)$,
\end{itemize}
so $f$ is a valid tiling, and obviously $f|_{[1,3]^2\cap \Z^2}=f'$.

Finally, each pathograph in $\Fc$ has no urpaths, and $\Hf$ has just one rung, completing the proof.
\end{proof}

\subsection{Vertex-colored pathographs}

In the second variation, we replace ``directed multicolored pathograph'' with ``vertex-colored pathograph,'' i.e.\ the edges, urpaths, spokes, and rungs will no longer be directed or colored; only the vertices are colored. The vertices on path $P_i$ in a realization $G$, replacing urpath $u_i$ of $\Hf$, are allowed to have any colors.

\begin{lemma}\label{lem:vcolor}
The following problem is undecidable: let $\Hf$ be a vertex-colored pathograph and $\Fc$ a finite set of vertex-colored pathographs; is there an $\Fc$-free realization of $\Hf$? In fact, it is undecidable even if $\Hf$ has just one rung and each pathograph in $\Fc$ has no urpaths.
\end{lemma}
\begin{proof}
Given an instance $(S, f')$ of the problem in Corollary~\ref{cor:specialundecidable}, let $\Hf,\Fc$ be the data constructed in the proof of Lemma~\ref{lem:dicolor}. We may assume $\abs{S}\ge 3$; otherwise add some dummy tiles to $S$ that are not allowed to be adjacent to anything. Let $K=\abs{S}$. Let $\Hf'$ be the vertex-colored pathograph, with colors from $\{1,\dots,K\}\cup\{-1,\dots,-K\}$, formed from $\Hf$ by the following process:
\begin{itemize}
    \item Replace each red vertex $x_i$ by a path $P_i$ (undirected, uncolored) of length $K$, whose vertices are named $x_i^{(1)},x_i^{(2)},\dots,x_i^{(K)}$ and colored $1,2,\dots,K$ in order.
    \item Replace each blue vertex $y_i$ by a path $Q_i$ (undirected, uncolored) of length $K$, whose vertices are named $y_i^{(1)},y_i^{(2)},\dots,y_i^{(K)}$ are colored $-1,-2,\dots,-K$ in order.
    \item Replace each red edge/urpath $x_ix_{i+1}$ with an edge/urpath between $x_i^{(K)}$ and $x_{i+1}^{(1)}$; replace each blue edge/urpath $y_iy_{i+1}$ with an edge/urpath between $y_i^{(K)}$ and $y_{i+1}^{(1)}$.
    \item Suppose the edge from $x_i$ to $y_j$ has color $t_k$. Replace it with a complete bipartite graph between the sets $\{x_i^{(\alpha)}\}_\alpha$ and $\{y_j^{(\alpha)}\}_\alpha$, minus one edge between $x_i^{(k)}$ and $y_j^{(k)}$.
    \item Forget the direction of all spokes and rungs.
\end{itemize}
See Figure~\ref{fig:vcolor} for a small example.

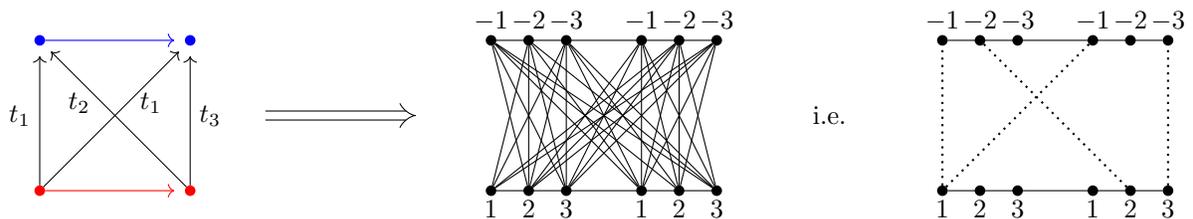
\begin{figure}[htbp!]
    \centering
    \begin{tikzpicture}[scale=2]
        \cvert{0}{0,0}{red};
        \cvert{1}{1,0}{red};
        \cvert{2}{0,1}{blue};
        \cvert{3}{1,1}{blue};

        \dicoledge{0}{1}{red};
        \dicoledge{2}{3}{blue};
        \draw[->, shorten >=4pt] (0) -- (2) node[midway, left] {$t_1$};
        \draw[->, shorten >=4pt] (0) -- (3) node[near end, below=0.05] {$t_1$};
        \draw[->, shorten >=4pt] (1) -- (2) node[near end, below=0.05] {$t_2$};
        \draw[->, shorten >=4pt] (1) -- (3) node[midway, right] {$t_3$};

        \draw[double distance=3pt, arrows={-Implies}, scaling nfold=2] (1.5, 0.5) -- (2.5, 0.5);

        \alabvert{0}{3,0}{1}{below};
        \alabvert{1}{3.25,0}{2}{below};
        \alabvert{2}{3.5,0}{3}{below};
        \alabvert{3}{4,0}{1}{below};
        \alabvert{4}{4.25,0}{2}{below};
        \alabvert{5}{4.5,0}{3}{below};

        \alabvert{6}{3,1}{-1}{above};
        \alabvert{7}{3.25,1}{-2}{above};
        \alabvert{8}{3.5,1}{-3}{above};
        \alabvert{9}{4,1}{-1}{above};
        \alabvert{10}{4.25,1}{-2}{above};
        \alabvert{11}{4.5,1}{-3}{above};

        \edge{0}{5};
        \edge{6}{11};

        \foreach \x in {0,...,5}
        \foreach \y in {6,...,11}
        \ifthenelse{\(\x=0 \AND \y=6\) \OR \(\x=0 \AND \y=9\) \OR \(\x=4 \AND \y=7\) \OR \(\x=5 \AND \y=11\)}{}{\edge{\x}{\y}};

        \node at (5.25, 0.5) {i.e.};
        
        \alabvert{0}{6,0}{1}{below};
        \alabvert{1}{6.25,0}{2}{below};
        \alabvert{2}{6.5,0}{3}{below};
        \alabvert{3}{7,0}{1}{below};
        \alabvert{4}{7.25,0}{2}{below};
        \alabvert{5}{7.5,0}{3}{below};

        \alabvert{6}{6,1}{-1}{above};
        \alabvert{7}{6.25,1}{-2}{above};
        \alabvert{8}{6.5,1}{-3}{above};
        \alabvert{9}{7,1}{-1}{above};
        \alabvert{10}{7.25,1}{-2}{above};
        \alabvert{11}{7.5,1}{-3}{above};

        \edge{0}{5};
        \edge{6}{11};

        \draw[dotted,thick] (0) -- (6);
        \draw[dotted,thick] (0) -- (9);
        \draw[dotted,thick] (4) -- (7);
        \draw[dotted,thick] (5) -- (11);
    \end{tikzpicture}
    \caption{Example translation from a directed multicolored pathograph to a vertex-colored pathograph. In this example, $K=3$. The labels of black edges/vertices indicate their color, when applicable. In the rightmost picture, we have just drawn the non-edges between the two sides, represented as dotted lines.}
    \label{fig:vcolor}
\end{figure}

Let $\Fc'$ be the following set of vertex-colored pathographs:
\begin{enumerate}
    \item Every graph formed by a path $x_1\cdots x_K$ where the color of $x_i$ is $i$ and a path $y_1\cdots y_K$ where the color of $y_i$ is $-i$ and any set of edges between the two paths except if those edges form a bipartite graph minus the edge $x_iy_i$ for some $i$.
    \item The graphs formed by applying the translation process to the type 2 graphs of $\Fc$.
    \item The graphs formed by applying the translation process to the type 3 graphs of $\Fc$.
\end{enumerate}

Let $\Fc''$ be $\Fc'$ plus the following vertex-colored pathographs (with no urpaths), continuing the numbering from the definition of $\Fc'$:
\begin{enumerate}[start=4]
    \item For each $i,j\in\{1,\dots,K\}$ with $i-j\not\equiv \pm 1\bmod K$, a pathograph consisting of one vertex of color $i$, one vertex of color $j$, and an edge between them.
    \item For each $i,j\in\{1,\dots,K\}$ with $i-j\not\equiv \pm 1\bmod K$, a pathograph consisting of one vertex of color $-i$, one vertex of color $-j$, and an edge between them.
    \item For each $i\in\{1,\dots,K\}$, a pathograph consisting of one vertex $v_1$ of color $i$ and two vertices $v_2,v_3$ of color $i+1$ (or $1$ if $i=K$), with edges $v_1v_2$ and $v_1v_3$.
    \item For each $i\in\{1,\dots,K\}$, a pathograph consisting of one vertex $v_1$ of color $-i$ and two vertices $v_2,v_3$ of color $-(i+1)$ (or $-1$ if $i=K$), with edges $v_1v_2$ and $v_1v_3$.
    \item For each $i,j\in\{1,\dots,K\}$ with $i\ne j$, a pathograph consisting of a vertex of color $i$, a vertex of color $-j$, and no edges.
\end{enumerate}

Now we claim that there is an $\Fc''$-free realization of $\Hf'$ if and only if there is an $\Fc$-free realization of $\Hf$.

($\Rightarrow$) Suppose there is an $\Fc''$-free realization $G'$ of $\Hf'$. Consider the colors of the vertices on the path $P=x_3^{(K)}p_1p_2\cdots p_\alpha x_1^{(1)}$ of $G'$, corresponding to the $x_3^{(K)}$-$x_1^{(1)}$ urpath in $\Hf'$. For each $\ell$, since $p_\ell$ is not adjacent to $x_2^{(1)}$ and not adjacent to $x_2^{(2)}$ (note there is no spoke from these vertices to the relevant urpath), its color must be positive due to the members of $\Fc''$ of type 8. Now the color of $p_1$ must be either $K-1$ or $1$ since it is adjacent to $x_3^{(K)}$ due to the members of $\Fc''$ of type 4; since $x_3^{(K)}$ is adjacent to $x_3^{(K-1)}$, the color of $p_1$ must actually be 1 due to the members of $\Fc''$ of type 6. By similar logic, the color of $p_\ell$ must be positive and congruent to $\ell$ modulo $K$ for all $\ell$. Likewise, for all $m$, the colors of the vertices $q_m$ on the path $P=y_3^{(K)}q_1q_2\cdots q_\beta y_1^{(1)}$ must all be negative and congruent to $-m$ modulo $K$, due to the members of $\Fc''$ of types 8, 5, and 7. Call this \textit{fact ($*$)}.

Now we form $G$, an $\Fc$-free realization of $\Hf'$, by the following process:
\begin{itemize}
    \item Replace each path $P=v_P^{(1)}\cdots v_P^{(K)}$ with colors $1,\dots,K$ in order by a single red vertex $x_P$. Call such paths $P$ \textit{red-forming paths}.
    \item Replace each path $Q=v_Q^{(1)}\cdots v_Q^{(K)}$ with colors $-1,\dots,-K$ in order by a single blue vertex $y_Q$. Call such paths $Q$ \textit{blue-forming paths}.
    \item Replace each edge between a vertex $v_P^{(K)}$ of color $K$ and a vertex $v_{P'}^{(1)}$ of color $1$ by a red directed edge from $v_P$ to $v_{P'}$.
    \item Replace each edge between a vertex $v_Q^{(K)}$ of color $-K$ and a vertex $v_{Q'}^{(1)}$ of color $-1$ by a blue directed edge from $v_Q$ to $v_{Q'}$.
    \item For each red-forming path $P$ and blue-forming path $Q$, let $v_P^{(\ell)}v_Q^{(\ell)}$ be the unique nonedge from $P$ to $Q$ in $G'$. Replace the bipartite graph between $P$ and $Q$ by a directed edge from $x_P$ to $y_Q$ of color $t_\ell$.
\end{itemize}
We must check two things to verify that this process is well-defined:
\begin{itemize}
    \item Each vertex of $G'$ is part of a unique red-forming or blue-forming path (red-forming if the color is positive and blue-forming if the color is negative). This is due to fact ($*$) above.
    \item For each pair of red-forming path $P$ and blue-forming path $Q$, there is exactly one nonedge from $P$ to $Q$, and this occurs between vertices of colors $\ell$ and $-\ell$ for some $\ell$. This is due to the members of $\Fc''$ of type 1.
\end{itemize}
Then it is easy to see that the resulting graph $G$ is indeed an $\Fc$-free realization of $\Hf'$. The fact it has no graphs from $\Fc$ of type 1 is clear from construction, and it has no graphs of types 2 and 3 in $\Fc$ because $G'$ has no graphs of types 2 and 3 in $\Fc''$.

($\Leftarrow$) Suppose that $G$ is an $\Fc$-free realization of $\Hf$. Recall the process that was used to convert $\Hf$ to $\Hf'$. Do this process on $G$ to obtain vertex-colored graph $G'$. It is easy to check that $G'$ is an $\Fc''$-free realization of $\Hf'$.

Finally, note that each pathograph in $\Fc''$ has no urpaths, and $\Hf'$ has just one rung, completing the proof.
\end{proof}

\subsection{Uncolored pathographs}

Now we can complete the proof of Theorem~\ref{thm:undecidable}.

\begin{proof}[Proof of Theorem~\ref{thm:undecidable}]
Given an instance $(S, f')$ of the problem in Corollary~\ref{cor:specialundecidable}, let $\Hf',\Fc''$ be data constructed in the proof of Lemma~\ref{lem:vcolor}. Let us assume $\abs{S}\ge 9$ (if not, add some dummy tiles to $S$ that are not allowed to be adjacent to any other tiles or themselves, which does not affect whether or not $S$ admits a periodic tiling). Let $K=\abs{S}$. Let $\Hf''$ be the pathograph constructed from $\Hf'$ by the following process:
\begin{itemize}
    \item Add a clique $C$ of size $3K$, say with vertices $c_1,\dots,c_{3K}$.
    \item For each $i\in \{1,\dots,2K\}$, add a vertex $z_i$ adjacent to $c_1,\dots,c_i$.
    \item For each vertex $v$ of color $i>0$, add an edge between $v$ and $z_i$, then forget the color of $v$.
    \item For each vertex $v$ of color $i<0$, add an edge between $v$ and $z_{K-i}$ (note that $K-i\in\{K+1,\dots,2K\}$), then forget the color of $v$.
    \item Add a spoke between each $z_i$ and urpath.
\end{itemize}
See Figure~\ref{fig:uncolored} for a small example.
\begin{figure}[htbp!]
    \centering
    \begin{tikzpicture}
    \alabvert{t1}{-2.5,-4}{1}{below};
    \alabvert{t2}{-1.5,-4}{2}{below};
    \alabvert{t3}{-0.5,-4}{3}{below};
    \alabvert{tn1}{0.5,-4}{-1}{below};
    \alabvert{tn2}{1.5,-4}{-2}{below};
    \alabvert{tn3}{2.5,-4}{-3}{below};
    \edge{t1}{t3};
    \edge{tn1}{tn3};

    \draw[double distance=3pt, arrows={-Implies}, scaling nfold=2] (3, -3.5) -- (5, -3.5);

    \tikzset{shift={(8,0)}};
    
    \foreach \x in {1,...,6} \vert{\x}{{-5.303*cos((\x+0.5)*11.25+45)},{-3*sin((\x+0.5)*11.25+45)+1.5}};
    \vert{7}{-1.25,-0.5};
    \vert{8}{0,-0.3};
    \vert{9}{1.25,-0.5};
    \foreach \x in {1,...,9} \foreach \y in {1,...,9} \ifthenelse{\x=\y}{}{\edge{\x}{\y}};
    \foreach \x in {1,...,6} \vert{a\x}{{\x - 3.5}, -3};
    \foreach \x in {1,...,6} \foreach \y in {1,...,\x} \edge{\y}{a\x};

    \vert{t1}{-2.5,-4};
    \vert{t2}{-1.5,-4};
    \vert{t3}{-0.5,-4};
    \vert{tn1}{0.5,-4};
    \vert{tn2}{1.5,-4};
    \vert{tn3}{2.5,-4};
    \edge{t1}{t3};
    \edge{tn1}{tn3};
    \edge{t1}{a1};
    \edge{t2}{a2};
    \edge{t3}{a3};
    \edge{tn1}{a4};
    \edge{tn2}{a5};
    \edge{tn3}{a6};

    \node[right] (z) at (3,-3) {$\{z_i\}_{i=1}^{2K}$};
    \node[right] (z) at (3,-0.9) {$\{c_j\}_{j=1}^{3K}=C$};

    \draw[thick, decorate, decoration={brace, amplitude=5pt}] (2.8, -2.75) -- (2.8, -3.25);
    \draw[thick, decorate, decoration={brace, amplitude=5pt}] (2.8, -0.05) -- (2.8, -1.75);
    \end{tikzpicture}
    \caption{Example translation from a vertex-colored pathograph to a (conventional) pathograph. Here we have $K=3$, but in the actual proof, we assume $K\ge 9$. As before, vertex labels in vertex-colored graphs represent colors. This example does not have any urpaths, but if it did, there should be a spoke from each $z_i$ to each urpath.}
    \label{fig:uncolored}
\end{figure}
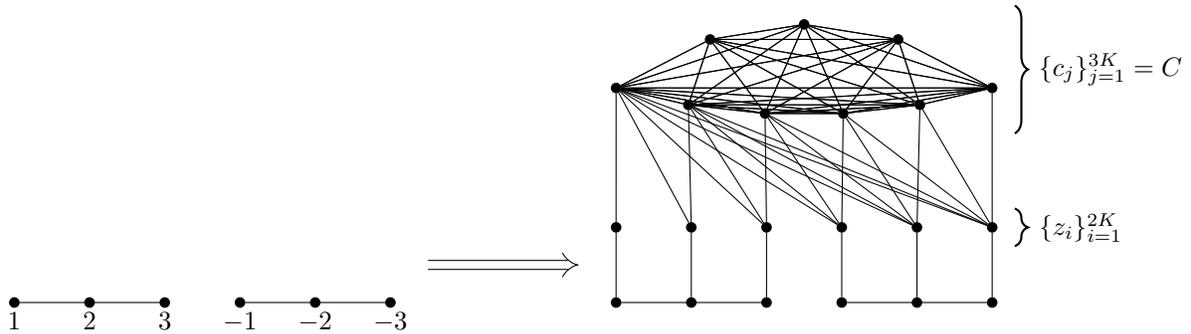

Do the same for all pathographs in $\Fc''$ to obtain the set $\Fc'''$, with the types of the graphs in $\Fc'''$ inherited from the types in $\Fc''$. Let $\Fc''''$ be $\Fc'''$ plus the following pathographs, continuing the numbering from the definition of $\Fc''$ (and $\Fc'''$):
\begin{enumerate}[start=9]
    \item For each $i,j\in \{1,\dots,2K\}$, a pathograph consisting of a clique $C$ of size $3K$ with vertices $c_1,\dots,c_{3K}$, vertices $z_i$ and $z_j$ adjacent to $\{c_1,\dots,c_i\}$ and $\{c_1,\dots,c_j\}$, respectively, and a vertex $v$ adjacent to $z_i$ and $z_j$.
    \item For each $i\in\{1,\dots,2K\}$, a pathograph consisting of a clique $C$ of size $3K$, a vertex $z_i$ adjacent to $i$ vertices of $C$, and a path $P=v_1\cdots v_{K+1}$ with $v_1$ adjacent to $z_i$.
\end{enumerate}

Now we claim that there is an $\Fc''''$-free realization of $\Hf''$ if and only if there is an $\Fc''$-free realization of $\Hf'$.

($\Rightarrow$) Suppose there is an $\Fc''''$-free realization $G''$ of $\Hf''$. Note that $G''$ may be partitioned into the following sets:
\begin{itemize}
    \item Two induced cycles $O_x,O_y$.
    \item A clique $C$ of size $3K$, say with vertices $c_1,\dots,c_{3K}$.
    \item $2K$ vertices $\{z_i\}_i$ adjacent to $C$, where $z_i$ is adjacent to $c_1,\dots,c_i$ for each $1\le i\le 2K$.
\end{itemize}
Every vertex of $O_x$ and $O_y$ must be adjacent to at most one $z_i$ due to the members of $\Fc''''$ of type 9. We will first prove that every vertex of $O_x$ and $O_y$ is in fact adjacent to exactly one $z_i$.

Note that $\Hf''$ has two urpaths, say $u_x$ between $x_3^{(K)}$ and $x_1^{(1)}$ and $u_y$ between $y_3^{(K)}$ and $y_1^{(1)}$. Let $P=x_3^{(K)}p_1\dots p_\alpha x_1^{(1)}$ be the path of $G''$ corresponding to $u_x$. Note that there is a path $x_3^{(1)}x_3^{(2)}\cdots x_3^{(K)}p_1$ of length $K$, and $x_3^{(i)}$ is adjacent to $z_i$ for each $i$. Due to the members of $\Fc''''$ of type 10, some member of this path other than $x_3^{(1)}$ must also be adjacent to $z_1$. But since all of the $x_3^{(i)}$, $i\ne 1$, are already adjacent to a $z_i$ that isn't $z_1$, the only possibility is that $p_1$ is adjacent to $z_1$. Repeating this argument with the path $x_3^{(2)}\cdots x_3^{(K)}p_1p_2$, we find that $p_2$ must be adjacent to $z_2$. This continues logic continues, so we find $p_\ell$ is adjacent to $z_{\ell'}$ where $\ell\equiv \ell'\bmod K$ and $\ell'\in\{1,\dots,K\}$. Similarly, if $Q=y_3^{(K)}q_1\dots q_\beta x_1^{(1)}$ is the path in $G''$ corresponding to the urpath $u_y$ of $\Hf''$, then each $q_m$ must be adjacent to $z_{K+m'}$ where $m\equiv m'\bmod K$ and $m'\in\{1,\dots,K\}$.

Let $G'$ be the vertex-colored graph obtained by replacing all of the vertices adjacent to $z_i$ by vertices of color $i$ if $i\le K$ and vertices of color $i-2K$ if $i > K$, then deleting $C$ and all $z_i$. It is easily verified that $G'$ is an $\Fc''$-free realization of $\Hf'$, as desired.

($\Leftarrow$) Suppose there is an $\Fc''$-free realization $G'$ of $\Hf'$. Recall the process used to form $\Hf''$ from $\Hf'$; perform this same process on $G'$ to obtain graph $G''$. We claim that $G''$ is an $\Fc''''$-free realization of $\Hf''$. The fact that it is a realization of $\Hf''$ is obvious.

Note that the only clique of size $3K$ in $G''$ is $C$. This may be seen by noting that $G''\setminus(C\cup\{z_i\}_i)$ is the (not disjoint) union of two induced cycles, so has chromatic number at most 6. The vertices $\{z_i\}_i$ form a stable set, so $G''\setminus C$ has chromatic number at most 7. But $K\ge 9$, so any clique of size $3K$ in $G''$ uses at least $3K-7$ vertices of $C$. Since any vertex not in $C$ has at most $2K<3K-8$ neighbors in $C$, any clique of size $3K$ must be $C$ itself.

Therefore, $G''$ contains no $\Ff'\in \Fc''''$ of types 1 through 8. Indeed, if it did, let $\Ff$ be the pathograph in $\Fc''$ corresponding to $\Ff'$. The clique of size $3K$ in $\Ff'$ must be mapped to $C$ in $G''$, and any vertices adjacent to the clique in $\Ff'$ must be mapped to $\{z_i\}_i$ in $G''$; let the rest of $\Ff'$ be mapped to vertex set $S\subseteq V(G'')\setminus (C\cup \{z_i\}_i)$. Then undoing the transformation from $G'$ to $G''$ by coloring the vertices in $G''$ according to which $z_i$ they are adjacent to, and deleting $C\cup \{z_i\}_i$, we find that $\Ff$ appears in $G'$ in the vertex set $S$.

By construction, $G''$ contains no member of $\Fc''''$ of type 9. Type 10 elements of $\Fc''''$ require more careful analysis. Recall that $G'''\coloneq G''\setminus(C\cup\{z_i\}_i)$ is the union of two induced cycles, say $O_x$ and $O_y$. We claim that all induced paths of length $K+1$ in $G'''$ are subsets of $O_x$ or $O_y$.

Let the vertices of $O_x$ be $x_1,\dots, x_\alpha$ and $O_y$ be $y_1,\dots, y_\beta$ in order (so $x_i$ is adjacent to $x_{i-1}$ and $x_{i+1}$ and $y_j$ is adjacent to $y_{j-1}$ and $y_{j+1}$, taking indices cyclically modulo $\alpha$ or $\beta$, respectively). Suppose $x_i$ is not adjacent to $y_j$. Then $i\equiv j\bmod K$, and $x_i$ must be adjacent to both $y_{j-1}$ and $y_{j+1}$, due to the members of $\Fc''$ of type 8. Suppose there is induced path $R$ of $H$ involving both vertices in $O_x$ and $O_y$. Then we claim this path can have length at most 6. There are two cases to consider:
\begin{itemize}
    \item Suppose two adjacent vertices of $R$ are both from $O_x$ or both from $O_y$. The existence of such vertices implies that there are three consecutive vertices $v_{i-1}v_iv_{i+1}$ of $R$ with either $v_{i-1},v_i\in O_x$ and $v_{i+1}\in O_y$, $v_{i-1},v_i\in O_y$ and $v_{i+1}\in O_x$, $v_{i-1}\in O_x$ and $v_i,v_{i+1}\in O_y$, or $v_{i-1}\in O_y$ and $v_i,v_{i+1}\in O_x$. All cases are symmetrical, so just consider the first case.
    
    Then $v_{i+2}$, if it exists, cannot be in $O_y$ or else it would be adjacent to $v_{i-1}$. So $v_{i+2}$ is in $O_x$, and $v_{i+3}$, if it exists, cannot be in $O_y$ or it would be adjacent to either $v_{i-1}$ or $v_i$. So $v_{i+3}$ is in $O_x$, and $v_{i+4}$ cannot possibly exist: if $v_{i+4}$ is in $O_x$ then it is adjacent to $v_{i+1}$, and if it is in $O_y$ then it is adjacent to either $v_{i-1}$ or $v_i$.

    Also, $v_{i-2}$, if it exists, cannot be in $O_x$ or it would be adjacent to $v_{i+1}$. So $v_{i-2}$ is in $O_y$. Then $v_{i-3}$ cannot exist: if $v_{i-3}$ is in $O_x$ then it would be adjacent to $v_{i+1}$, and if $v_{i-3}$ is in $O_y$ then it would be adjacent to either $v_{i-1}$ or $v_i$.

    So in this case, $R$ can contain at most 6 vertices: $v_{i-2}$ through $v_{i+3}$.
    \item Suppose no two adjacent vertices of $R$ are both from $O_x$ or both from $O_y$. Suppose $v_1\in R$ is in $O_x$; the other case is symmetrical. Then $v_6$ cannot exist (where the vertices of $R$ are labeled $v_1,\dots,v_{\abs{R}}$ in order). This is because $v_4$ is adjacent to $v_3$ but not $v_1$, so if $v_1=x_i$, it must be that $v_3=x_j$ with $i\not\equiv j\bmod K$. Then $v_6$ is adjacent to either $v_1$ or $v_3$.

    So in this case, $R$ can contain at most 5 vertices.
\end{itemize}

Thus the only way that an element of $\Fc''''$ of type 10 can appear is if the needed path of length $K+1$ in $G''$ lies entirely in $O_x$ or entirely in $O_y$, since $K\ge 9$. Also, one must note that this path cannot contain any of the $z_i$ in $G''$, since those are all adjacent to at least one vertex of $C$, and if an element $\Ff\in \Fc''''$ of type 10 appears in $G''$, the clique of size $3K$ in $\Ff$ must be mapped to $C$ in $G''$, as we have previously shown. But by construction, every $K$th vertex of $O_x$ (or $O_y$) is adjacent to the same $z_i$, i.e.\ if $v_1\cdots v_{K+1}$ is an induced path in $G''$ and $v_1$ is adjacent to $z_i$, then $v_{K+1}$ is also adjacent to $z_i$. We conclude that $G''$ is $\Fc''''$-free.

One can easily verify that $\Hf''$ has only one rung and no $\Ff\in\Fc''''$ have any urpaths, completing the proof of Theorem~\ref{thm:undecidable}.
\end{proof}

\subsection{Related problems}
\label{sec:variants}

We remark that what we have really proven is the following:

\paths*

\begin{proof}
Let $\Hf$ and $\Fc$ be the data constructed in the proof of Theorem~\ref{thm:undecidable}. Let $u_1$ and $u_2$ be the two urpaths of $\Hf$. Let $a$ and $b$ be the endpoints of $u_1$ and $c$ and $d$ be the endpoints of $u_2$, and let $U_1$ and $U_2$ be the sets of vertices adjacent to $u_1$ and $u_2$ (via spokes), respectively. Then let $H$ be the graph formed by deleting the urpaths of $\Hf$, with chosen vertices $a,b,c,d$ and sets of vertices $U_1$ and $U_2$. Then it is not difficult to see that the problem described in the statement of the corollary applied to the data $(\Fc, H, a, b, c, d, U_1, U_2)$ is exactly the same as the pathograph realization problem applied to the data $(\Hf, \Fc)$.
\end{proof}
In other words, our result may be reformulated using conventional graph terminology, with no mention of pathographs.

We now discuss two restricted forms of the pathograph realization problem. First, for a fixed finite set of pathographs $\Fc$, we consider the following problem:
\begin{problem}\label{prob:restrict1}
Given a pathograph $\Hf$, is there an $\Fc$-free realization of $\Hf$?
\end{problem}
Second, for a fixed pathograph $\Hf$, we consider:
\begin{problem}\label{prob:restrict2}
Given a finite set of pathographs $\Fc$, is there an $\Fc$-free realization of $\Hf$?
\end{problem}

To be clear, each finite set of pathographs gives a different variant of Problem~\ref{prob:restrict1}. The input to that problem is just the pathograph $\Hf$. Likewise the input to Problem~\ref{prob:restrict2} is just $\Fc$; there is a different variant for each $\Hf$. One can prove the following:
\begin{theorem}\label{thm:restricted}
There exists a finite set of pathographs $\Fc$ so that the variant of Problem~\ref{prob:restrict1} with that $\Fc$ is undecidable. There exists a pathograph $\Hf$ so that the variant of Problem~\ref{prob:restrict2} with that $\Hf$ is undecidable.
\end{theorem}

In order to prove this, we will need to encode Turing machines as Wang tiles. There are many ways to do this, it is done in~\cite{berger} for example. For our purposes, we will simply need an encoding with properties according to the following proposition, which we will not construct explicitly. First, we give a bit more terminology. A \textit{section} of a Turing machine tape is simply a contiguous subset of the tape, and a \textit{patch} of a tiling is a rectangular subset of the tiling. Say a Turing machine $M$ \textit{strongly loops} on input $X$ if $M$ eventually returns the tape exactly to the input $X$, with the tape head in the starting position and machine state at the starting state. It is obviously undecidable to determine if a given Turing machine strongly loops on a given input, by an easy reduction of the usual halting problem to this problem (or by appealing to Rice's theorem).
\begin{proposition}\label{prop:encoding}
Let $M$ be a Turing machine. There is a set of Wang tiles $S_M$ and a function $f_M$ that sends finite patches of Turing machine tape to finite patches using tiles from $S_M$ so that $S_M$ admits a periodic tiling containing the patch $f(X)$ if and only if $M$ strongly loops on (finite) input $X$.
\end{proposition}

We will also need the following undecidability result.
\begin{lemma}\label{lem:bespoke}
There exists a Turing machine $M^*$ with the following property. It is undecidable to determine, given a set $\Tc$ of tape sections, if there is any input $X$ so that $M^*$ loops on $X$ and no tape section in $T$ ever appears on the tape as $M^*$ runs on input $X$.
\end{lemma}
\begin{proof}
Let us say $(M, X)$ is \textit{$\Tc$-good} if $M$ strongly loops on $X$ and the tape never contains any tape section in $T$. We wish to prove that for some machine $M^*$, given $\Tc$ it is undecidable to determine if there is any $X$ so that $(M^*, X)$ is $T$-good.

First, for any fixed set $\Tc$, by Rice's theorem one of the following must be the case:
\begin{enumerate}
    \item Every Turing machine $M$ has an input $X$ so that $(M, X)$ is $\Tc$-good.
    \item No Turing machine $M$ has an input $X$ so that $(M, X)$ is $\Tc$-good.
    \item It is undecidable to determine if a given machine $M$ has any input $X$ so that $(M, X)$ is $\Tc$-good.
\end{enumerate}
The first case cannot occur since there are Turing machines that do not strongly loop on any input. There are obviously sets $\Tc$ such that at least one machine $M$ and input $X$ is $\Tc$-good; by process of elimination, such $\Tc$ must lie in the third case. So there is some $\Tc^*$ so that determining if a given $M$ has any input $X$ such that $(M, X)$ is $\Tc^*$-good.

Let $U$ be a universal Turing machine with the following properties. The input to $U$ is an encoding of a Turing machine $M$ and input $X$. The required properties are as follows:
\begin{enumerate}
    \item The symbols used by $U$ consist of the blank symbol $\#$, a special separator symbol $\%$ and three sets of symbols: $I$ consisting of symbols used to encode the input machine $M$ (including the blank symbol used by $M$, which is different than $\#$), $A$ consisting of symbols used by $M$, and $B$ consisting of some extra symbols to by used by $U$ to perform the computation.
    \item The tape used by $U$ may be partitioned as $LCR$ (for ``left'', ``center'', ``right''), so that at all stages in the computation, $L$ consists only of symbols in $\{\#,\%\}\cup B$, $C$ consists only of symbols in $I$, and $R$ consists only of symbols in $\{\#,\%\}\cup A\cup B$.
    \item Suppose that when $M$ is run on input $X$, at step $i$ the tape is $W_i$ padded by infinitely many blank symbols on each side, where the first and last symbols of $W_i$ are not blank. Then when $U$ is run on input $(M, X)$, at each step:
    \begin{enumerate}
        \item The half-tape $R$ is either equal to $\%W_i\#^\infty$ for some $i$, or it is equal to $\%YW\#^\infty$ where $Y$ consists only of symbols in $B$ and $W$ is a strict (possibly empty) substring of $W_i$ for some $i$.
        \item For each $i$, $R$ is at some point equal to $\%W_i\#^\infty$.
        \item $C$ is constant (so at all times it just contains the encoding of $M$).
        \item $L$ contains exactly one instance of $\%$, which is the last symbol of $L$.
    \end{enumerate}
    \item The input $(M, X)$ to $U$ is provided as a tape section of the form $\% C_M\% X$ where $C_M$ encodes $C$. $U$ strongly loops on this input if and only if $M$ strongly loops on input $X$. $U$ halts on an input of any other form.
\end{enumerate}

Now we claim that this $U$ is the $M^*$ we desire. To see this, first consider our special set of tape sections $\Tc^*$, which only uses symbols in $B$. For a given $M$, let $C_M$ be the encoding of $M$ in the input to $U$ and $\Tc_M$ be the set of the following tape sections:
\begin{enumerate}
    \item All tape sections of the form $\% Z\%$ where $Z$ is length at most the length of $C_M$, $Z\ne C_M$, and $Z$ uses only symbols from $I$.
    \item All tape sections of the form $Z$ where $Z$ is length longer than the length of $C_M$ and $Z$ uses only symbols from $I$.
\end{enumerate}

Now one can check that the following are equivalent by the properties of $U$ and the construction of the set $\Tc_M$:
\begin{itemize}
    \item There is an input $Y$ so that $(U, Y)$ is $(\Tc^*\cup \Tc_M)$-good.
    \item There is an input $X'$ so that $(U, \%C_M\% X')$ is $\Tc^*$-good.
    \item There is an input $X$ so that $(M, X)$ is $\Tc^*$-good.
\end{itemize}

Since $\Tc^*$ was chosen so that it is undecidable to determine if there is such an $X$ for a given $M$ in the third statement above, it is undecidable to determine, for a given set $\Tc$, if there is any input $Y$ to $U$ so that $(U, Y)$ is $\Tc$-good, as desired.
\end{proof}

Now let us prove Theorem~\ref{thm:restricted}.
\begin{proof}[Proof of Theorem~\ref{thm:restricted}]
Let $U$ be a universal Turing machine with the following property. The input to $U$ is an encoding of a Turing machine $M$ and input $X$; the property is that $U$ strongly loops on input $(M, X)$ if and only if $M$ strongly loops on input $X$. Let $S_U$ be a set of Wang tiles that simulates $U$ according to Proposition~\ref{prop:encoding}, and $f_U$ the associated function that translates sections of Turing machine tape to patches of tiles.

Let us suppose for a given $(M, X)$ that our tile patch $P=f_U(M, X)$ is a $k\times \ell$ patch of tiles, with $k,\ell\ge 3$ depending on $(M, X)$. Now convert this patch of tiles $P$ to a directed multicolored pathograph that looks like the pathograph in Figure~\ref{fig:dimultiH}, except that the two directed triangles are replaced by directed cycles of length $k$ and $\ell$, with the edges between the two sides encoding the patch of tiles $P$. Then follow the construction in the proof of Theorem~\ref{thm:undecidable} starting from that $\Hf$. So, since we fixed our set of tiles $S_U$, we have that the resulting $\Hf''$ depends on $P$ (which depends on $M$ and $X$), while the resulting $\Fc''''$ will be the same every time (since $S_U$ is fixed). We have that the following are equivalent, by the construction:
\begin{itemize}
    \item $\Hf''$ admits an $\Fc''''$-free realization.
    \item $S$ admits a periodic tiling containing $P$.
    \item $U$ strongly loops on input $(M, X)$.
    \item $M$ strongly loops on input $X$.
\end{itemize}
Since the last problem in this list is undecidable, the first problem is also undecidable. This proves the first assertion of this theorem for the variant of Problem~\ref{prob:restrict1} corresponding to $\Fc''''$.

For the second assertion, now let $M^*$ be the Turing machine constructed in Proposition~\ref{lem:bespoke}. Let $S_{M^*}$ be a set of Wang tiles simulating $M^*$ in the sense of Proposition~\ref{prop:encoding}, and $f_{M^*}$ the associated translation function. Let $\Hf''$ and $\Fc''''$ be the data constructed in the proof of Theorem~\ref{thm:undecidable} for this $S'$. Given a tape section $W$, we first convert it to a tile patch $f_{M^*}(W)$, then apply a similar construction to the one in the proof of Theorem~\ref{thm:undecidable} to obtain a graph $G(W)$, so that the following are equivalent:
\begin{itemize}
    \item $\Hf''$ admits an $\Fc''''\cup\{G(W)\}$-free realization.
    \item $S_{M^*}$ admits a periodic tiling not containing $f_{M^*}(W)$.
    \item There is an input $X$ so that $M^*$ strongly loops on $X$ and the tape never contains the tape section $W$.
\end{itemize}

Applying the same process to a whole set $\Tc$ of tape sections, we have that the following are equivalent:
\begin{itemize}
    \item $\Hf''$ admits an $\Fc''''\cup\{G(W)\}_{W\in \Tc}$-free realization.
    \item $S_{M^*}$ admits a periodic tiling not containing $f_{M^*}(W)$ for any $W\in \Tc$.
    \item There is an input $X$ so that $M^*$ strongly loops on $X$ and the tape never contains any tape section $W\in \Tc$.
\end{itemize}

By construction of $M^*$, the last problem on this list is undecidable, so the first problem is also undecidable, proving the second assertion of this theorem for the variant of Problem~\ref{prob:restrict2} corresponding to $\Hf''$.
\end{proof}

\section{The rungless case}
\label{sec:rungless}

In this section, we prove the following:

\rungless*
\lineartime*

In the latter theorem, we are assuming that the input realization $G$ to this algorithm uses vertex labels compatible with the vertices of $\Hf$, or in other words, this algorithm does not need to compute \textit{how} $G$ is a realization of $\Hf$; that information must be given in the input. This is a necessary restriction, since the problem of determining if a given graph is a realization of $\Hf$ can itself be NP-complete (depending on what $\Hf$ is). We assume the reader is familiar with basic properties of regular languages and finite automata, which can be found in any standard text on computability theory, e.g.\ \cite{sipser}.

We give two proofs of these theorems. The first is an application of (the proof of) Courcelle's theorem. This uses the more advanced notion of a \textit{tree automaton}; for a reference on this subject, refer to \cite{tata} or the various references in Courcelle's paper \cite{courcelle}. The second proof, given in the appendix, constructs an explicit conventional finite automata to accomplish the same thing as Courcelle's tree automaton.

For the first proof, we will also need the notions of \textit{pathwidth} and \textit{treewidth}. A \textit{tree decomposition} of a graph $G$ is a pair $(T,\chi)$ where $T$ is a tree and $\chi:V(T)\to 2^{V(G)}$ is a function so that:
\begin{enumerate}
    \item For each $v\in V(G)$, there is a $t\in V(T)$ so that $v\in \chi(t)$.
    \item For each $uv\in E(G)$, there is a $t\in V(T)$ so that $u,v\in \chi(t)$.
    \item Let $\chi^{-1}(v)$ be the set of $t\in V(T)$ so that $v\in \chi(t)$. Then $\chi^{-1}(v)$ induces a connected subgraph of $T$ for each $v$.
\end{enumerate}

A \textit{path decomposition} of $G$ is a tree decomposition where $T$ is a path. The \textit{treewidth} of $G$ is the minimum over all tree decompositions $(T,\chi)$ of $G$ of the maximum of $\abs{\chi(t)} - 1$ for $t\in V(T)$. The \textit{pathwidth} of $G$ is the minimum over all path decompositions of the same quantity.

Note that path graphs all have pathwidth 1, the pathwidth of a disjoint union of graphs $G$ and $H$ is the maximum of the pathwidths of $G$ and $H$, the pathwidth of a (not necessarily disjoint) union of $G$ and $H$ is at most the sum of the pathwidths of $G$ and $H$ plus 1, and the treewidth of a graph $G$ is at most the pathwidth of $G$.

A graph property is \textit{monadic second-order} if it may be expressed as a boolean formula with quantifiers over vertices, edges, subsets of vertices, and subsets of edges, and we have the predicates expressing that a vertex is in a particular set of vertices, an edge is in a particular set of edges, and a vertex is incident to an edge.

We have the important theorem relating tree decompositions and monadic second-order properties:

\begin{theorem*}[Courcelle's theorem~\cite{courcelle}]
Testing if $G$ has a given monadic second-order property $P$ can be done in time $f(P)\abs{V(G)}$ if a bounded-width tree decomposition of $G$ is provided as input, for some function $f$.
\end{theorem*}

The dependence on the property to be tested is quite poor; the function $f$ is roughly a power tower whose height is the number of quantifier alternations in the property. The standard proof is by constructing a tree automaton that accepts a tree decomposition of bounded width if and only if the graph has the desired property. It is this result that we need.

\begin{proof}[Proof of Theorems~\ref{thm:rungless} and~\ref{thm:lineartime}]
Note that any realization of $\Hf$, if $\Hf$ has no rungs, may be expressed as the union of $K$ induced paths with no edges between them, plus at most $N$ other vertices, where $K$ is the number of urpaths in $\Hf$ and $N$ is the number of vertices in $\Hf$. Therefore realizations of $\Hf$ have bounded pathwidth; in particular, pathwidth at most $K+N$. This also means realizations of $\Hf$ have bounded treewidth. Additionally, the property of being $\Fc$-free, for a finite set of pathographs $\Fc$, is easily seen to be monadic second-order. To see if a particular $\Ff\in \Fc$ appears in a realization $G$ of $\Hf$, one only needs to search for the existence of some number of vertices and some number of sets of vertices, so that the chosen sets induce paths and the adjacencies between all of the vertices and sets of vertices matches the edges, spokes, and rungs of $\Ff$. The property of being a realization of $\Hf$ is also monadic second-order, by similar logic.

It then follows that we may construct tree automaton $M$ that takes as input a bounded-width tree decomposition of a graph and accepts if and only if the graph is an $\Fc$-free realization of $\Hf$. By standard results about tree automata, it is decidable to determine if there is any tree accepted by $M$, which proves Theorem~\ref{thm:rungless}. Additionally, Theorem~\ref{thm:lineartime} is proved by constructing a bounded-width tree decomposition of a realization of $G$ (which may easily be done in linear time since the vertex labels of $G$ were assumed to be compatible with $\Hf$), then feeding that tree decomposition to $M$, which works in linear time.
\end{proof}

Some readers may find this proof unsatisfying due to its reliance on Courcelle's theorem. In Appendix~\ref{app:explicit}, we give an explicit (conventional) automaton construction to the same end, including some small examples of how parts of this automaton look.

\subsection{Nuanced characterizations}
\label{ssec:misc}

Suppose in Problem~\ref{prob:realization} that $\Hf$ is not $\Fc$-free, i.e.\ it contains some $\Ff\in \Fc$. Then obviously no realizations of $\Hf$ can be $\Fc$-free. Theorem~\ref{thm:undecidable} shows indirectly that the converse does not hold: there are pathographs $\Hf$ that are $\Fc$-free but have no $\Fc$-free realizations. Here is a simple explicit example of this fact. Let $\Hf$ be the pathograph in Figure~\ref{fig:H}. It is tedious but not difficult to check that all realizations of $\Hf$ contain a theta, prism, or wheel; i.e.\  if one sets $\Fc=\Theta\cup \Pr\cup \W$ (from Theorem~\ref{thm:encode}), then the answer to Problem~\ref{prob:realization} is ``no'' despite the fact that $\Hf$ is $\Fc$-free. Since $\Hf$ has no rungs, this observation can be proven completely mechanically using the algorithm described in the previous subsection.

\begin{figure}[htbp!]
    \centering
    \begin{tikzpicture}
    \useasboundingbox (-1,-0.2) -- (1,2.2);
    \labvert{c}{0,0}{below};
    \labvert{a}{0,2}{above};
    \labvert{d}{1,1}{right};
    \labvert{b}{-1,1}{left};
    
    \urpath{a}{c};
    \edge{a}{d};
    \edge{a}{b};
    \edge{c}{d};
    \edge{c}{b};
    \spoke{d}{a}{c};
    \spoke{b}{a}{c};
    \end{tikzpicture}
    \captionof{figure}{The example pathograph $\Hf$.}
    \label{fig:H}
\end{figure}

Continuing with this example, suppose that we wish to characterize the theta- and wheel-free realizations of $\Hf$. Intuitively, this set should be fairly restricted, since all members contain a prism, but just how restricted is it? Since $\Hf$ has no rungs, the set of theta- and wheel-free realizations of $\Hf$ form a regular language in some sense. Using the ``determination string'' notation defined in Appendix~\ref{app:explicit}, one can verify that this language is given by the regular expression\footnote{We have suppressed the first index in each symbol (which is always 1 since $\Hf$ has only one urpath), e.g.\ $\{a,b\}$ should be interpreted as $(1, \{a,b\})$.}
\begin{equation}\label{eq:reglang}
    (\{a,b\}\varnothing^*\{c,d\})\mid (\{a,d\}\varnothing^*\{c,b\})
\end{equation}
This claim can be proven automatically by checking that the automaton $M$ constructed in Theorem~\ref{thm:automaton} recognizes precisely this language in the case $\Fc=\Theta\cup\W$. This regular language can then be reinterpreted in terms of pathographs: all $(\Theta\cup\W)$-free realizations of $\Hf$ are actually realizations of one of the pathographs in Figure~\ref{fig:thetawheelfree}, so they do not just contain prisms but in fact are prisms with at least two of the three paths having length 1.

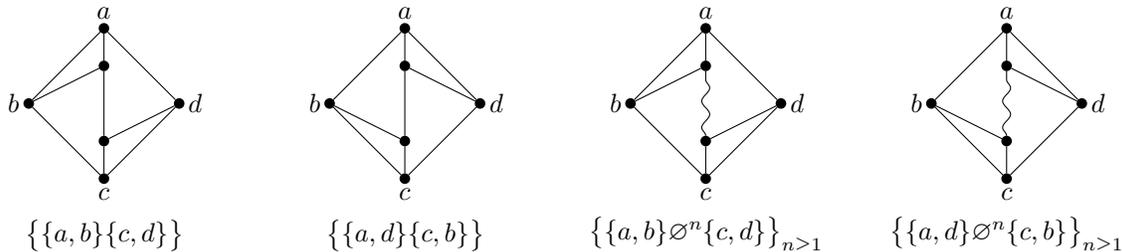
\begin{figure}[htbp!]
    \centering
    \begin{tikzpicture}
    \labvert{c}{0,0}{below};
    \labvert{a}{0,2}{above};
    \labvert{d}{1,1}{right};
    \labvert{b}{-1,1}{left};
    \vert{e}{0,0.5};
    \vert{f}{0,1.5};
    
    \edge{e}{f};
    \edge{e}{d};
    \edge{f}{b};
    \edge{e}{c};
    \edge{f}{a};
    \edge{a}{d};
    \edge{a}{b};
    \edge{c}{d};
    \edge{c}{b};

    \node at (0,-0.75) {$\big\{\{a,b\}\{c,d\}\big\}$};

    \tikzset{shift={(4,0)}};
    \labvert{c}{0,0}{below};
    \labvert{a}{0,2}{above};
    \labvert{d}{1,1}{right};
    \labvert{b}{-1,1}{left};
    \vert{e}{0,0.5};
    \vert{f}{0,1.5};
    
    \edge{e}{f};
    \edge{e}{b};
    \edge{f}{d};
    \edge{e}{c};
    \edge{f}{a};
    \edge{a}{d};
    \edge{a}{b};
    \edge{c}{d};
    \edge{c}{b};

    \node at (0,-0.75) {$\big\{\{a,d\}\{c,b\}\big\}$};

    \tikzset{shift={(4,0)}};
    \labvert{c}{0,0}{below};
    \labvert{a}{0,2}{above};
    \labvert{d}{1,1}{right};
    \labvert{b}{-1,1}{left};
    \vert{e}{0,0.5};
    \vert{f}{0,1.5};
    
    \urpath{e}{f};
    \edge{e}{d};
    \edge{f}{b};
    \edge{e}{c};
    \edge{f}{a};
    \edge{a}{d};
    \edge{a}{b};
    \edge{c}{d};
    \edge{c}{b};

    \node at (0,-0.75) {$\big\{\{a,b\}\varnothing^n\{c,d\}\big\}_{n\ge 1}$};

    \tikzset{shift={(4,0)}};
    \labvert{c}{0,0}{below};
    \labvert{a}{0,2}{above};
    \labvert{d}{1,1}{right};
    \labvert{b}{-1,1}{left};
    \vert{e}{0,0.5};
    \vert{f}{0,1.5};
    
    \urpath{e}{f};
    \edge{e}{b};
    \edge{f}{d};
    \edge{e}{c};
    \edge{f}{a};
    \edge{a}{d};
    \edge{a}{b};
    \edge{c}{d};
    \edge{c}{b};

    \node at (0,-0.75) {$\big\{\{a,d\}\varnothing^n\{c,b\}\big\}_{n\ge 1}$};
    \end{tikzpicture}
    \caption{The set of pathographs that describe the $(\Theta\cup\W)$-free realizations of the pathograph $\Hf$ from Figure~\ref{fig:H}. Beneath each pathograph, we have written the corresponding set of determination strings; the union of these sets is the language \eqref{eq:reglang}. The first and second pathographs are both isomorphic to $\Pr_1$, but they yield different determination strings, and similarly for the third and fourth pathographs, which are both isomorphic to $\Pr_2$.}
    \label{fig:thetawheelfree}
\end{figure}

Hence the techniques in this section cannot only decide Problem~\ref{prob:realization} if $\Hf$ has no rungs, but provide a characterization of all $\Fc$-free realizations of $\Hf$ in terms of a regular language.

\subsection{Applications to decomposition theorems}
\label{ssec:decomposition}

As mentioned in the introduction, the rungless case of Problem~\ref{prob:realization} occurs quite frequently when proving decomposition theorems in various ``attachment lemmas''. We give a specific example. In~\cite{uniquechord}, Claim 3 in the proof of Theorem 2.3 is the following:

\begin{lemma}
Suppose $G$ contains the Petersen graph plus a path between two different vertices of the Petersen graph. Then $G$ contains a triangle, an induced $C_4$, or a cycle with a unique chord.
\end{lemma}

This lemma can, in principle, be proven using the techniques in this section. In particular, such a graph $G$ is a realization of a pathograph formed by adding one urpath to the Petersen graph, or, if the two endpoints are adjacent, adding a two-edge path or a one-urpath~+~one-edge path, plus any number of spokes (but notably no rungs). The condition that $G$ does not contain a cycle with a unique chord is simply that $G$ does not have any of the pathographs in Figure~\ref{fig:uniquechord}. So then the lemma statement can be reformulated as a pathograph realization problem with no rungs, and solved (in principle) by Theorem~\ref{thm:rungless}.

\begin{figure}[htb!]
    \centering
    \begin{tikzpicture}
        \vert{1}{-1,0};
        \vert{2}{1,0};
        \vert{3}{0,1};
        \vert{4}{0,-1};
        \edge{1}{2};
        \edge{1}{3};
        \edge{2}{3};
        \edge{1}{4};
        \edge{2}{4};
    \end{tikzpicture}\qquad
    \begin{tikzpicture}
        \vert{1}{-1,0};
        \vert{2}{1,0};
        \vert{3}{0,1};
        \vert{4}{0,-1};
        \edge{1}{2};
        \edge{1}{3};
        \urpath{2}{3};
        \edge{1}{4};
        \edge{2}{4};
    \end{tikzpicture}\qquad
    \begin{tikzpicture}
        \vert{1}{-1,0};
        \vert{2}{1,0};
        \vert{3}{0,1};
        \vert{4}{0,-1};
        \edge{1}{2};
        \edge{1}{3};
        \urpath{2}{3};
        \edge{1}{4};
        \urpath{2}{4};
    \end{tikzpicture}
    \caption{The pathographs associated with having a cycle with a unique chord.}
    \label{fig:uniquechord}
\end{figure}
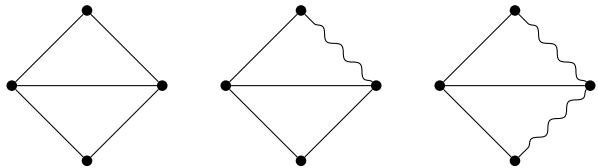

This example came about from analyzing how a path can attach onto a Petersen graph. It was amenable to the techniques in this section because the structure we were attaching onto is a conventional graph, i.e.\ a graph with no urpaths, so adding just one urpath can never introduce rungs. Another common step in proving decomposition theorems is to analyze how a single \textit{vertex} can attach onto a structure. These proof steps are always amenable to our techniques if the structure the vertex is being attached to is given by a pathograph with no rungs, since adding a vertex can only introduce edges and spokes.

That said, the finite automaton constructed in the proof of Theorem~\ref{thm:rungless}, even the explicit one provided in the appendix, will usually be very large, so this sort of analysis might not be practical without some further optimizations or insight about the specific problem being considered.

\section{Forbidden structures closed under adding adjacencies}\label{sec:closed}

We identify one more decidable case of Problem~\ref{prob:realization}. Suppose that $\Fc=\cl_\sim(\Fc)$, with $\cl_\sim$ defined in the proof of Theorem~\ref{thm:encode}. In other words, $\Fc$ is closed under adding edges, spokes, and rungs. Let us say that $\Fc$ is \textit{closed} for brevity. Such sets can arise naturally from $\Fc$ defining a set of forbidden subgraphs, minors, and/or topological minors. If $\Fc$ is closed, then the property of being $\Fc$-free is monotone under adding edges, spokes, and rungs: if $\Gf$ is not $\Fc$-free then no graph in $\cl_\sim(\Gf)$ is $\Fc$-free.

Say a realization $G$ of $\Hf$ is \textit{minimal} if removing any edge from $G$ makes it no longer a realization. We can solve Problem~\ref{prob:realization} if $\Fc$ is closed using the observation that, since being $\Fc$-free is monotone, there is an $\Fc$-free realization of $\Hf$ if and only if some minimal realization of $\Hf$ is $\Fc$-free.

\closeddecidable*

\begin{proof}
We prove this by induction on the number of rungs in $\Hf$. The case of no rungs is Theorem~\ref{thm:rungless}.

Let $\{u_1,u_2\}$ be a rung in $\Hf$. Let $S$ be the set of pathographs $\Hf'$ formed from $\Hf$ by the following process:
\begin{itemize}
    \item Say the urpath $u_1$ has endpoints $a_1$ and $b_1$. Delete $u_1$ and add exactly one of the following:
    \begin{itemize}
        \item A vertex $c_1$ and edges from $c_1$ to $a_1$ and $b_1$.
        \item A vertex $c_1$, edge from $c_1$ to $a_1$, and urpath $c_1^b$ from $c_1$ to $b_1$.
        \item A vertex $c_1$, edge from $c_1$ to $b_1$, and urpath $c_1^a$ from $c_1$ to $a_1$.
        \item A vertex $c_1$ and urpaths from $c_1$ to $a_1$ and $b_1$ called $c_1^a$ and $c_1^b$, respectively.
    \end{itemize}
    Do the same for $u_2$, obtaining new vertex $c_2$ and possibly new urpaths $c_2^a$ and/or $c_2^b$.
    \item Add an edge between $c_1$ and $c_2$.
    \item For each spoke $(v,u_1)$, add exactly one of the following:
    \begin{itemize}
        \item An edge between $v$ and $c_1$.
        \item A spoke between $v$ and $c_1^a$, if it exists.
        \item A spoke between $v$ and $c_1^b$, if it exists.
    \end{itemize}
    Do the same for $u_2$.
    \item For each rung $\{u_1,u'\}$, add exactly one of the following:
    \begin{itemize}
        \item A spoke between $c_1$ and $u'$.
        \item A rung between $u'$ and $c_1^a$, if it exists.
        \item A rung between $u'$ and $c_1^b$, if it exists.
    \end{itemize}
    Do the same for $u_2$.
\end{itemize}

Each pathograph in $S$ has strictly fewer rungs than $\Hf$: the rung $\{u_1,u_2\}$ disappeared, and all other rungs were replaced by either one rung or one spoke. Also, $S$ is clearly finite.

We claim that $G$ is a minimal realization of $\Hf$ if and only if $G$ is a minimal realization of some $\Hf'\in S$. Indeed, if $G$ is a non-minimal realization of $\Hf$, then some spoke or rung in $\Hf$ is represented by at least 2 edges in $G$. If this spoke/rung is not incident to $u_1$ or $u_2$, then $G$ is obviously not a minimal realization of any $\Hf'$. If this spoke/rung is incident to $u_1$ or $u_2$, then $G$ is still not a minimal realization of any $\Hf'$ since each such spoke/rung was replaced by exactly one edge, spoke, or rung in each $\Hf'$. This proves one direction, and the other direction (minimal realizations of each $\Hf$ are also minimal realizations of some $\Hf'$) is obvious.

By inductive hypothesis, Problem~\ref{prob:realization} is decidable with input $\Hf'$ and $\Fc$, for any $\Hf'\in S$. Then the answer to Problem~\ref{prob:realization} with input $\Hf$ and $\Fc$ is ``yes'' if and only if there is an $\Fc$-free realization of $\Hf'$ for some $\Hf'\in S$, since $\Fc$ is closed. This completes the inductive step and the proof.
\end{proof}

\section{Acknowledgments}

We thank Maria Chudnovsky for helpful discussions and Laurent Viennot for suggesting Problem~\ref{prob:restrict1}. The second author is supported by Projet ANR GODASse, Projet-ANR-24-CE48-4377.

\printbibliography

\appendix

\section{Explicit automaton construction}
\label{app:explicit}

In this appendix we give a second proof of Theorems~\ref{thm:rungless} and~\ref{thm:lineartime} using an explicit automaton construction, including some small examples. Our construction results in a fairly large automaton in general, though it can certainly be optimized in places.

Theorems~\ref{thm:rungless} and~\ref{thm:lineartime} will follow from the theorem below. For a pathograph $\Hf$ with no rungs and realization $G$ of $\Hf$ (with compatible vertex labels), with paths $P_i$ in $G$ replacing urpaths $u_i$ in $\Hf$, we say the \textit{external neighborhood string} of a path $P_i=a_iv_i^{(1)}\cdots v_i^{(k)}b_i$ is the string 
\[ s(P_i) = \sum_j \left(N(v_i^{(j)})\cap V(\Hf)\right) \in (2^{V(\Hf)})^* \]
over the alphabet $\Sigma_1\coloneqq 2^{V(\Hf)}$, where by sum we mean concatenation. We say the \textit{determination string} of $G$ is the string
\[ \sigma(G) = \sum_i \sum_j (i, s(P_i)_{(j)}) \in (\{1,\dots,K\}\times 2^{V(\Hf)})^* \]
over the alphabet $\Sigma_2\coloneqq \{1,\dots,K\}\times 2^{V(\Hf)}$, where $\Hf$ has $K$ urpaths. Here, by $s(P_i)_{(j)}$ we mean the $j$th symbol of $s(P_i)$, i.e.\ $N(v_i^{(j)})\cap V(\Hf)$. It is easy to see that $G$ may be uniquely reconstructed from its determination string.
\begin{theorem}\label{thm:automaton}
For any pathograph $\Hf$ with no rungs and finite set of pathographs $\Fc$, there is a (deterministic) finite automaton $M$ (depending on both $\Hf$ and $\Fc$) that takes in as input $\sigma\in (\Sigma_2)^*$ and accepts $\sigma$ if and only if $\sigma$ is the determination string of an $\Fc$-free realization of $\Hf$.
\end{theorem}

First we show that this implies Theorems~\ref{thm:rungless} and~\ref{thm:lineartime}.

\begin{proof}[Proof of Theorem~\ref{thm:rungless} assuming Theorem~\ref{thm:automaton}]
Let $L_M\subseteq (\Sigma_2)^*$ be the regular language given by the strings accepted by $M$. It is a standard result that determining if a regular language is empty is decidable; the answer to Problem~\ref{prob:realization} is ``no'' if and only if $L_M=\varnothing$.
\end{proof}

\begin{proof}[Proof of Theorem~\ref{thm:lineartime} assuming Theorem~\ref{thm:automaton}]
Given $G$, one can obviously find the determination string $\sigma(G)$ in linear time (assuming the vertex labels of $G$ are compatible with those of $\Hf$). Then we simply feed this string into $M$, which operates in linear time; $G$ is $\Fc$-free if and only if $M$ accepts $\sigma(G)$.
\end{proof}

\subsection{Partial pathograph inclusions}

To prove Theorem~\ref{thm:automaton}, we first introduce the notion of a \textit{partial pathograph inclusion} between pathographs $\Hf$ and $\Gf$. We say that $\phi=(\phi_V, \phi_P)$ is a partial pathograph inclusion if:
\begin{itemize}
    \item $\phi_V:V(\Hf)\to V(\Gf)\cup\{\ud\}$ is injective except that $\phi_V$ may send multiple vertices of $\Hf$ to the special $\ud$ element.
    \item $\phi_U:U(\Hf)\to 2^{V(\Gf)\cup P(\Gf)}\cup\{\ud\}$ has the following properties:
    \begin{itemize}
        \item If $\phi_U(u)\ne \ud$, then the elements of $\phi_U(u)$ are disjoint and nonadjacent paths (possibly of length 1, i.e.\ vertices) in $\Gf$.
        \item Suppose $u\in U(\Hf)$ has endpoints $v_1,v_2\in V(\Hf)$. Then $\phi_V(v_1)$ and $\phi_V(v_2)$ are endpoints of some elements of $\phi_U(u)$, or are $\ud$. If $\phi_U(u)$ is just one path and $\phi_V(v_1)$ and $\phi_V(v_1)$ are not $\ud$, we say that $\phi_U(u)$ is a \textit{completed path}.
        \item If $\phi_U(u)=\ud$ then both endpoints $v_1$ and $v_2$ of the urpath $u$ must have $\phi_V(v_1)=\phi_V(v_2)=\ud$ as well.
    \end{itemize}
    \item Let $a,b\in V(\Hf)\cup U(\Hf)$. Then $a$ and $b$ are adjacent in $\Hf$ if and only if $\phi(a)$ and $\phi(b)$ are adjacent in $\Gf$, with the following exceptions:
    \begin{itemize}
        \item if $\phi(a)$ or $\phi(b)$ are $\ud$, or
        \item if $a$ (or $b$) is a path and $\phi(a)$ (or $\phi(b)$) is not a completed path.
    \end{itemize}
\end{itemize}
We write $\phi:\Hf\to\Gf\cup\{\ud\}$ to denote that $\phi$ is a partial pathograph inclusion. We say that $\phi$ \textit{extends} to a pathograph inclusion $\psi:\Hf\to\Gf$ if $\psi_V(v)=\phi_V(v)$ whenever $\phi_V(v)$ is not $\ud$ and $\psi_U(u)\supseteq \bigcup_{q\in \phi_U(u)}q$ whenever $\phi_U(u)$ is not $\ud$.

The bulk of the work in proving Theorem~\ref{thm:automaton} is relegated to the following lemma:
\begin{lemma}
Let $\Hf$ be a pathograph without rungs and $\Ff$ a pathograph. Let $H$ be the graph formed by deleting all urpaths from $\Hf$ and $\eta:H\to \Hf$ the associated pathograph inclusion map. Suppose $\phi:\Ff\to H\cup\{\ud\}$ is a partial pathograph inclusion. There is a finite automaton $M_\phi$ that takes as input $\sigma\in (\Sigma_2)^*$ and:
\begin{itemize}
    \item Suppose $\sigma$ is the determination string of a realization $G$ of $\Hf$, and let $\xi:\Hf\to G$ be the associated pathograph inclusion map (see Proposition~\ref{prop:basicfacts}(2)). Then $\sigma$ is accepted if and only if $\phi$ extends to a pathograph inclusion $\phi':\Ff\to H$ such that, defining $\psi=\xi\circ\eta\circ\phi':\Ff\to G$, we have $\im(\psi_V)\setminus \im(\phi_V)\subseteq V(G)\setminus V(H)$ and $\psi_U(u)\setminus \bigcup_{q\in \phi_U(u)}q\subseteq V(G)\setminus V(H)$.
    \item If $\sigma$ is not the determination string of a realization of $\Hf$, then $\sigma$ may be either accepted or rejected.
\end{itemize}
\end{lemma}
Another way to interpret the first bullet is as follows. We may view $\phi:\Ff\to H\cup\{\ud\}$ as a partial pathograph inclusion $\tilde{\phi}:\Ff\to G\cup\{\ud\}$ by composing with $\eta$ and $\xi$. This automaton $M_\phi$ must accept $\sigma(G)$ if and only if there is a way to extend $\tilde{\phi}$ to a pathograph inclusion $\Ff\to G$ using only new vertices in $G$ that were not present in $H$. For the time being, we do not care about the behavior of $M_\phi$ on strings that are not determination strings. We will combine these machines $M_\phi$ to construct the machine $M$ described in Theorem~\ref{thm:automaton}.

\begin{proof}
There are finitely many ``essentially different'' ways to extend $\phi$ to such a pathograph inclusion $\psi$. We will demonstrate this by example before making it precise. Suppose $u$ is an urpath of $\Ff$ with endpoints $v_1',v_2'\in V(\Ff)$ and $\phi_U(u)=\{q_1,q_2,q_3\}$ with $v_1,a_1$ the endpoints of $q_1$ (a path in $G$), $v_2,a_2$ the endpoints of $q_2$, $a_3,b_3$ the endpoints of $q_3$, and $\phi_V(v_i')=v_i$ for $i=1,2$. If $\phi$ is to extend to a pathograph inclusion $\psi:\Ff\to G$, it must be that $\psi_U(u)$ consists of one of the following paths in $G$:
\begin{itemize}
    \item The concatenation of $q_1$, then a path from $a_1$ to $a_3$, then $q_3$, then a path from $b_3$ to $a_2$, then $q_2$.
    \item The concatenation of $q_1$, then a path from $a_1$ to $b_3$, then $q_3$ in reverse order, then a path from $a_3$ to $a_2$, then $q_2$.
\end{itemize}

See Figure~\ref{fig:partialpath} for an illustration. In the first case, we say we are ``searching for'' an $a_1$-$a_3$ path $p_1$ and a $b_3$-$a_2$ path $p_2$ in $G\setminus H$; in the second case, we are searching for an $a_1$-$b_3$ path and an $a_3$-$a_2$ path.

\begin{figure}[htbp!]
    \centering
    \begin{tikzpicture}
        \alabvert{v1}{0,0}{v_1}{below};
        \alabvert{a1}{1,0}{a_1}{below};
        \alabvert{a3}{2,0}{a_3}{below};
        \alabvert{b3}{3,0}{b_3}{below};
        \alabvert{a2}{4,0}{a_2}{below};
        \alabvert{v2}{5,0}{v_2}{below};
        \labdotedge{v1}{a1}{q_1}{above};
        \labdotedge{v2}{a2}{q_2}{above};
        \labdotedge{a3}{b3}{q_3}{above};
        \node[left] (phiu) at (-1,0) {$\phi_U(u)\subset H$};

        \alabvert{v1F}{0,2}{v_1'}{below};
        \alabvert{v2F}{5,2}{v_2'}{below};
        \laburpath{v1F}{v2F}{u}{above};
        \node[left] (u) at (-1,2) {$u\subset \Ff$};

        \alabvert{v11}{0,-2}{v_1}{below};
        \alabvert{a11}{1,-2}{a_1}{below};
        \alabvert{a31}{2,-2}{a_3}{below};
        \alabvert{b31}{3,-2}{b_3}{below};
        \alabvert{a21}{4,-2}{a_2}{below};
        \alabvert{v21}{5,-2}{v_2}{below};
        \labdotedge{v11}{a11}{q_1}{above};
        \labdotedge{v21}{a21}{q_2}{above};
        \labdotedge{a31}{b31}{q_3}{above};
        \labdotedge{a11}{a31}{p_1}{above};
        \labdotedge{b31}{a21}{p_2}{above};

        \alabvert{v12}{0,-4}{v_1}{below};
        \alabvert{a12}{1,-4}{a_1}{below};
        \alabvert{a32}{2,-4}{a_3}{below};
        \alabvert{b32}{3,-4}{b_3}{below};
        \alabvert{a22}{4,-4}{a_2}{below};
        \alabvert{v22}{5,-4}{v_2}{below};
        \labdotedge{v12}{a12}{q_1}{above};
        \labdotedge{v22}{a22}{q_2}{above};
        \labdotedge{a32}{b32}{q_3}{above};
        \draw[dashed] (a12) to[bend right=60] node[midway, below] {$p_1$} (b32);
        \draw[dashed] (a32) to[bend left=60] node[midway, above] {$p_2$} (a22);

        \node[left] (psiu) at (-1,-3) {$\psi_U(u)\subset G$};

        \draw[thick, decorate, decoration={brace, mirror, amplitude=5pt}] (-0.75, -1.75) -- (-0.75, -4.25);
    \end{tikzpicture}
    \caption{The two ways of extending $\phi$ to a pathograph inclusion $\psi$, from the perspective of the image of the urpath $u$. Here, dashed lines in $G$ represent paths.}
    \label{fig:partialpath}
\end{figure}
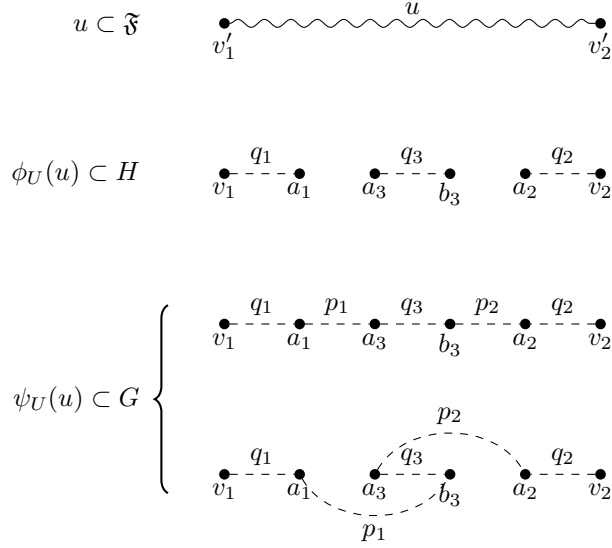

Let us just focus on the first case. We must additionally have that $p_1$ is not adjacent to $v_1$, $q_1$, $q_3$, $b_3$, $p_2$, $a_2$, $q_2$, or $v_2$. Likewise, $p_2$ must not be adjacent to $v_1$, $q_1$, $a_1$, $p_1$, $a_3$, $q_3$, $q_2$, or $v_2$. There are additional adjacency or nonadjacency constraints on $p_1$ and $p_2$ related to the rest of $\im(\phi)$; for example, $p_1$ and $p_2$ must be nonadjacent to $\phi_V(v)$ if $v$ is a vertex not adjacent to $u$ in $\Ff$ and $\phi_V(v)\ne\ud$.

The two different cases above represent ``essentially different'' ways of extending $\phi$ to a pathograph inclusion, but this is not the full story. Note that $V(G)\setminus V(H)$ is a disjoint union of paths, so $p_1$ and $p_2$ must be found within these paths. Even within the two cases identified, there are still ``essentially different'' ways of finding the necessary paths, depending on where in $G$ they are found. For example, if $\Hf$ has two urpaths $w_1$ and $w_2$, corresponding to paths $P_1$ and $P_2$ in $G$, then possibly $p_1$ and $p_2$ could both be found in $P_1$, or $p_1$ could be in $P_1$ and $p_2$ could be in $P_2$, or so on.

\textit{Crucially}, $p_1$ cannot be found by including vertices from both $P_1$ and $P_2$ since $\Hf$ has no rungs (so the urpaths $w_1$ and $w_2$ are not adjacent, and $P_1$ and $P_2$ are not adjacent in $G$). There is further distinction within some of these subcases depending on order that $p_1$ and $p_2$ appear; for instance, if both $p_1$ and $p_2$ are to be found in $P_1$, then either $p_1$ is closer to the left endpoint of $P_1$, or $p_2$ is. They cannot be ``interlaced'' in any fashion because $P_1$ is an induced path so all connected induced subgraphs of $P_1$ are given by ``contiguous'' subsets of $P_1$. ++Each of these possible locations to find $p_1$ and $p_2$ will count as an ``essentially different'' way of extending $\phi$.

We can repeat the same analysis on all urpaths $u$ of $\Ff$ where $\phi(u)$ is not a completed path, and a similar process on all vertices $v$ of $\Ff$ where $\phi(v)=\ud$. Each possibility consists of:
\begin{itemize}
    \item the description of all paths and vertices in $G\setminus H$ that we are searching for,
    \item all adjacency/nonadjacency constraints on these paths/vertices,
    \item which path of $G$ we find these paths/vertices in (recalling the crucial observation above), and
    \item the order within each path that these paths/vertices are found.
\end{itemize}
This complete set of data is one ``essentially different'' way to extend $\phi$ to a pathograph inclusion, and it is easy to see there are only finitely many such possibilities.

\vspace{1em}

Let $D$ be the data associated to one such possibility; specifically, $D$ consists of the following data:
\begin{itemize}
    \item A (finite) set of pairs $p_1\coloneqq(a_1,b_1),\dots,p_\alpha\coloneqq(a_\alpha,b_\alpha)$, where each $a_i$ and $b_i$ is a vertex in $H$ (i.e.\ a vertex in $G$ outside of the paths that replaced the urpaths in $\Hf$), so that for each $i$ we need to find an $a_i$-$b_i$ path, and so that the $a_i$ end of the path is found before the $b_i$ end of the path.
    \item A (finite) set of symbols $x_1,\dots,x_\beta$. For each $j$, $x_j$ will represent $\psi(v_j)$ where $v_j\in V(\Ff)$ has $\phi(v_j)=\ud$; i.e.\ the $x_j$ represent the vertices we must find in $G\setminus H$ to extend $\phi_V$.
    \item A subset of $\binom{V(H)\cup \{p_i\}_i\cup \{x_j\}_j}{2}$; the presence of a pair $\{y,z\}$ means that $y$ and $z$ must be adjacent in $G$, and the absence of such a pair means that $y$ and $z$ must be nonadjacent in $G$.
    \item A function $f:\{p_i\}_i\cup \{x_j\}_j\to \{1,\dots,K\}$ where $\{w_k\}_{k=1}^K$ are the urpaths of $\Hf$. Suppose that $P_k$ is the path in $G$ replacing urpath $w_k$ of $\Hf$, for each $k$. The meaning of $f(y)$ is that element $y$ (either a path or vertex) must be found in path $P_{f(y)}$ in $G$.
    \item For each $k$, an order $<_k$ on $f^{-1}(k)$, giving the order within $P_k$ that we find each of the necessary paths $p_i$ and vertices $x_j$.
\end{itemize}

Then we claim there is a finite automaton $M_D$ taking strings $\sigma\in (\Sigma_2)^*$ and, if $\sigma$ is the determination string of $G$, accepts $\sigma$ if and only if $\phi$ extends to a pathograph inclusion in $G$ in exactly the manner prescribed by the data $D$. We will construct $M_D$ as a nondeterministic automaton.

For each $y\in \{p_i\}_i\cup \{x_j\}_j$, let $A_y\subseteq V(H)\cup \{p_i\}_i\cup \{x_j\}_j$ be the set of objects that $y$ must be adjacent to. Note that we may assume $A_y$ contains at most two elements of $\{p_i\}_i\cup \{x_j\}_j$; otherwise, it is impossible for $\phi$ to be extended to a pathograph inclusion in the prescribed way since the elements of $\{p_i\}_i\cup \{x_j\}_j$ adjacent to $y$ must be found in the same path $P_{f(y)}$ as $y$, and there are only two locations such an element could appear (either immediately before $y$ or immediately after, in the sense of $<_{f(y)}$). The automaton $M_D$ can simply reject all inputs in the case three or more such elements appear in $A_y$. Call the \textit{type} of $y$ the number of elements of $A_y$ that are in $\{p_i\}_i\cup \{x_j\}_j$ (i.e.\ each $y$ may be type 0, type 1, or type 2).

Call a type 1 object $y$ \textit{type 1a} if the element of $A_y\cap (\{p_i\}_i\cup \{x_j\}_j)$ must come before $y$ (as indicated by $<_{f(y)}$) and \textit{type 1b} otherwise. For $y$ that are not the minimum under $<_{f(y)}$, let $\ell_y$ be the element of $f^{-1}(f(y))$ immediately before $y$ under $<_{f(y)}$; if $y$ is not the maximum under $<_{f(y)}$, let $r_y$ be the element immediately after $y$. Let $A_y'=A_y\setminus (\{p_i\}_i\cup \{x_j\}_j)$ be the elements of $A_y$ that are not $\ell_y$ or $r_y$.

Now we construct $M_D$. The states of $M_D$ are as follows:
\begin{itemize}
    \item For each $k\in \{1,\dots,K\}$, a state $s_k^{\text{start}}$. The start state is $s_1^{\text{start}}$.
    \item For each $p_i$ and $X\subseteq A_{p_i}'$, a state $s_{p_i}^X$.
    \item For each $y\in \{p_i\}_i\cup \{x_j\}_j$ that is type 0 or type 1a, a state $s_y^{\text{right}}$ and a state $s_y^{\text{after}}$.
    \item For each $y\in \{p_i\}_i\cup \{x_j\}_j$ that is type 1b or type 2, a state $s_y^{\text{right}}$.
    \item A state $s_{\text{accept}}$, which is the only accepting state.
\end{itemize}
The transitions of $M_D$ are as follows:
\begin{itemize}
    \item For each $k$ and $\lambda\in\Sigma_1$, there is a transition from $s_k^{\text{start}}$ to itself with label $(k, \lambda)$ (which is in $\Sigma_2$).
    \item For each $\lambda\in\Sigma_1$, there is a transition from $s_{\text{accept}}$ to itself with label $(K, \lambda)$.
    \item For each $y$ of type 0 or type 1a and $\lambda\in \Sigma_1$, there is a transition from $s_y^{\text{after}}$ to itself with label $(f(y), \lambda)$.
    \item For each $k$, if $f^{-1}(k)$ is empty, there is an $\varepsilon$-transition from $s_k^{\text{start}}$ to $s_{k+1}^{\text{start}}$ unless $k=K$, in which case there is an $\varepsilon$-transition from $s_k^{\text{start}}$ to $s_{\text{accept}}$.
    \item For each $k$, if $f^{-1}(k)$ is nonempty, let $y_k\in \{p_i\}_i\cup \{x_j\}_j$ be the smallest element of $f^{-1}(k)$ under $<_k$ and $y_k'$ be the largest element. Necessarily $y_k$ is type 0 or type 1b and $y_k'$ is type 0 or type 1a. Then:
    \begin{itemize}
        \item If $y_k=p_i$, then for each $X\subseteq A_{y_k}'$, there is a transition from $s_k^{\text{start}}$ to $s_{y_k}^X$ with label $(f(y_k), X\cup \{a_i\})$ (where $a_i$ is the left endpoint of $p_i=(a_i,b_i)$). Additionally, there is a transition from $s_k^{\text{start}}$ to $s_{y_k}^{\text{right}}$ with label $(f(y_k), A_{y_k}'\cup \{a_i, b_i\})$ ($b_i$ is the right endpoint of $p_i$)
        \item If $y_k=x_j$, then there is a transition from $s_k^{\text{start}}$ to $s_{y_k}^{\text{right}}$ with label $(f(y_k), A_{y_k}')$.
        \item If $k<K$, there is an $\varepsilon$-transition from $s_{y_k'}^{\text{right}}$ to $s_{k+1}^{\text{start}}$ and an $\varepsilon$-transition from $s_{y_k'}^{\text{after}}$ to $s_{k+1}^{\text{start}}$; if $k=K$, these transitions instead go to $s_{\text{accept}}$.
    \end{itemize}
    \item For each $p_i$ and $X\subseteq A_{p_i}'$, if $\ell_{p_i}$ exists:
    \begin{itemize}
        \item If $\ell_{p_i}$ is type 0 or type 1a (equivalently, $p_i$ is type 0 or type 1b), there is a transition from $s_{\ell_{p_i}}^{\text{after}}$ to $s_{p_i}^X$ with label $(f(p_i), X\cup\{a_i\})$.
        \item If $\ell_{p_i}$ is type 1b or type 2 (equivalently, $p_i$ is type 1a or type 2), there is a transition from $s_{\ell_{p_i}}^{\text{right}}$ to $s_{p_i}^X$ with label $(f(p_i), X\cup\{a_i\})$.
    \end{itemize}
    \item For each $p_i$ and $X,Y\subseteq A_{p_i}'$, there is a transition from $s_{p_i}^X$ to $s_{p_i}^{X\cup Y}$ with label $(f(p_i), Y)$.
    \item For each $p_i$ and $X,Y\subseteq A_{p_i}'$, if $X\cup Y=A_{p_i}'$, there is a transition from $s_{p_i}^X$ to $s_{p_i}^{\text{right}}$ with label $(f(p_i), Y\cup \{b_i\})$.
    \item For each $y$ of type 0 or type 1a and $\lambda\in\Sigma_1$, there is a transition from $s_y^{\text{right}}$ to $s_y^{\text{after}}$ with label $(f(y), \lambda)$.
    \item For each $x_j$, if $\ell_{x_j}$ exists:
    \begin{itemize}
        \item If $\ell_{x_j}$ is type 0 or type 1a, there is a transition from $s_{\ell_{x_j}}^{\text{right}}$ to $s_{x_j}^{\text{right}}$ with label $(f(x_j), A_{x_j}')$.
        \item If $\ell_{x_j}$ is type 1b or type 2, there is a transition from $s_{\ell_{x_j}}^{\text{after}}$ to $s_{x_j}^{\text{right}}$ with label $(f(x_j), A_{x_j}')$.
    \end{itemize}
\end{itemize}

Note that if we are in state $s$ and read symbol $\kappa\in\Sigma_2$, but there is no transition from $s$ with label $\kappa$, then the string is rejected. Additionally, note that $M_D$ is nondeterministic, so there may be multiple legal transitions that could be taken in any given step.

One can check that if $\sigma\in\Sigma_2^*$ is the determination string of $G$, then $M_D$ accepts $\sigma$ if and only if $\phi$ extends to a pathograph inclusion $\Ff\to G$ in exactly the manner prescribed by $D$. Checking this is exceptionally tedious, so we just informally describe the main subclaims to verify:
\begin{itemize}
    \item If the current state is $s_{p_i}^X$, then we are currently constructing the path $p_i$ and so far we have ``accumulated'' the adjacencies $X$ out of the required set $A_{p_i}'$.
    \item If the current state is $s_y^{\text{right}}$, then we have just successfully completed constructing $y$ (either a path or a vertex), unless $y$ is type 1b or type 2, in which case we have constructed everything except for the adjacency to $r_y$, which must be constructed using the very next symbol.
    \item If the current state is $s_y^{\text{after}}$, then we have constructed $y$ and read at least one more symbol, so the next symbol we read represents a vertex that is not adjacent to $y$.
    \item If the current state is $s_k^{\text{start}}$, we are ready to read a symbol of the form $(k, \lambda)$, i.e.\ we are revealing the path $P_k$ in $G$ but have not yet started constructing any of the objects in $f^{-1}(k)$.
    \item If the current state is $s_{\text{accept}}$, we have successfully constructed every object we are looking for.
\end{itemize}

Now we use $M_D$ to construct the desired automaton $M_\phi$. To do this, let $L_D$ be the regular language given by the strings accepted by $M_D$. Then construct $M_\phi$ as an automaton to recognize the regular language
\[ \bigcup_{\text{data }D} L_D. \]
This language is regular since it is a finite union of regular languages.
\end{proof}

Before completing the proof of Theorem~\ref{thm:automaton}, we give an example of the construction in the proof above. Let $\Hf$ be the pathograph in Figure~\ref{fig:H}.

In this case, $H$ is the four-cycle with vertices labeled $a,b,c,d$ in order. We will reveal the $a$-$c$ path in realizations $G$ from the $a$ end to the $c$ end (not including the vertices $a$ and $c$, of course). Let $\Ff$ the pathograph in Figure~\ref{fig:labeledwheel}.

\begin{figure}
    \centering
    \begin{tikzpicture}
    \useasboundingbox (0,-0.2) -- (2,2.2);
    \labvert{X}{0,1}{left};
    \labvert{Z}{1,1}{above};
    \labvert{Y}{2,1}{right};
    \draw[decorate,decoration={snake,amplitude=1.5,segment length=10, post length=0,pre length=0}] (X) to[bend right=90] node[midway, below=0.1] {$u_2$} (Y);
    \draw[decorate,decoration={snake,amplitude=1.5,segment length=10, post length=0,pre length=0}] (X) to[bend left=90] node[midway, above=0.1] {$u_1$} (Y);
    \edge{X}{Z};
    \edge{Y}{Z};
    \manspoke{Z}{1,0.25};
    \end{tikzpicture}
    \captionof{figure}{The example pathograph $\Ff$.}
    \label{fig:labeledwheel}
\end{figure}

Let $\phi:\Ff\to H$ be given by
\begin{align*}
    \phi_V(X) &= b \\
    \phi_V(Y) &= d \\
    \phi_V(Z) &= \ud \\
    \phi_U(u_1) &= \{bad\} \\
    \phi_U(u_2) &= \{b,d\}.
\end{align*}
Suppose the data $D$ is the following:
\begin{itemize}
    \item The symbol $p_1=(b,d)$, representing $\psi_U(u_2)$. In this case, we have determined that we will find the $b$ end of the path before the $d$ end.
    \item The symbol $x_1$, representing $\psi_U(Z)$.
    \item The set
    \begin{align*}
        \big\{&\{a,b\}, \{b,c\}, \{c,d\}, \{a,d\}, \\
        &\{p_1,x_1\}, \{x_1,b\}, \{x_1,c\}, \{x_1,d\}\big\}
    \end{align*}
    representing adjacencies. Note that we have decided that $x_1$ will be adjacent to $c$ and $p_1$ will not be adjacent to $c$.
    \item The function $f:\{p_1,x_1\}\to\{1\}$ that just maps everything to 1 (there is only one urpath in $\Hf$).
    \item The order $<_1$ with $p_1<_1x_1$, indicating that we will find $p_1$ before $x_1$.
\end{itemize}
We have that $p_1$ is type 1b and $x_1$ is type 1a. Then $M_D$ is the automaton shown in Figure~\ref{fig:dataauto}.
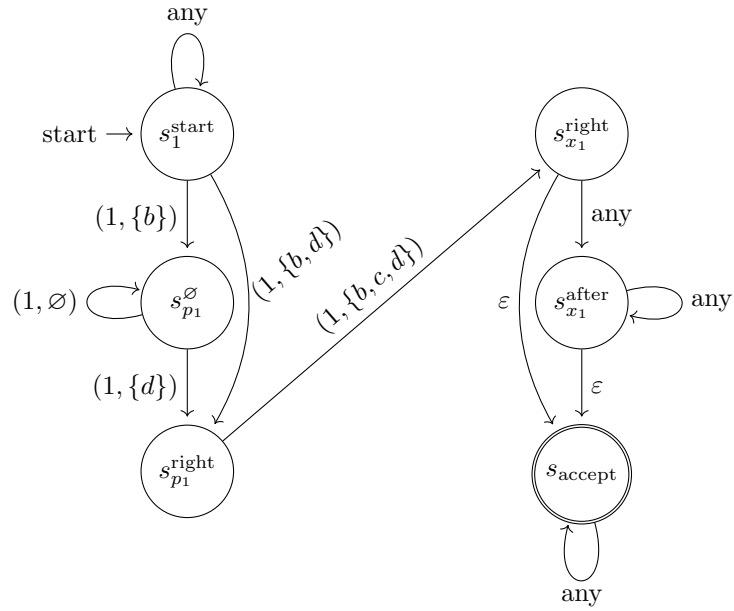
\begin{figure}[htbp!]
    \centering
    \begin{tikzpicture}[shorten >=3pt, state/.style={circle, draw, minimum size=3.5em}, auto, transform shape]
        \node[state, initial] (1start) {$s_1^{\text{start}}$};
        \node[state, below = of 1start] (p1n) {$s_{p_1}^\varnothing$};
        \node[state, below = of p1n] (p1r) {$s_{p_1}^{\text{right}}$};
        \node[state, right = 4 of 1start] (x1r) {$s_{x_1}^{\text{right}}$};
        \node[state, below = of x1r] (x1a) {$s_{x_1}^{\text{after}}$};
        \node[state, accepting, below = of x1a] (acc) {$s_{\text{accept}}$};

        \path[->]
        (1start) edge [loop above] node {any} (1start)
        (1start) edge [swap] node {$(1, \{b\})$} (p1n)
        (1start) edge [bend left, rotate=41] node {$(1, \{b,d\})$} (p1r)
        (p1n) edge [loop left] node {$(1, \varnothing)$} (p1n)
        (p1n) edge [swap] node {$(1, \{d\})$} (p1r)
        (p1r) edge [above, rotate=41] node {$(1, \{b, c, d\})$} (x1r)
        (x1r) edge node {any} (x1a)
        (x1r) edge [bend right, swap] node {$\varepsilon$} (acc)
        (x1a) edge [loop right] node {any} (x1a)
        (x1a) edge node {$\varepsilon$} (acc)
        (acc) edge [loop below] node {any} (acc)
        ;
    \end{tikzpicture}
    \caption{The automaton $M_D$ for the example data. An arrow labeled ``any'' indicates that transition may be taken by reading any symbol. A doubly-circled state is accepting. The start state is $s_1^{\text{start}}$. Roughly speaking, the three states on the left side of the diagram search for $p_1$, and the three states on the right search for $x_1$.}
    \label{fig:dataauto}
\end{figure}

This construction can be repeated for all possible data $D$ given this partial pathograph inclusion $\phi$, and the resulting automata can be combined to an automaton $M_\phi$. In Figure~\ref{fig:phiauto}, we give one possible (nondeterministic) $M_\phi$ with the required properties.

\begin{figure}[htbp!]
    \centering
    \begin{tikzpicture}[shorten >=3pt, state/.style={circle, draw, minimum size=3.5em}, auto, transform shape, node distance=2]
    \node[state] (s) {$s_{\text{start}}$};
    \node[above left = 0.4 of s] (0) {start};
    \node[state, below = of s] (x) {$s_{p_1}^{(b)}$};
    \node[state, below right = of x] (xy1) {$s_{p_1}^{\text{right}}$};
    \node[state, above right = of xy1] (y) {$s_{p_1}^{(d)}$};
    \node[state, below left = of x] (cy) {$s_{p_1\mid x_1}^{(d)}$};
    \node[state, above left = of cy] (c1) {$s_{x_1}^{\text{right}}$};
    \node[state, below left = of c1] (cx) {$s_{p_1\mid x_1}^{(b)}$};
    \node[state, accepting, below = of cy] (a) {$s_{\text{accept}}$};
    
    \path[->]
    (0) edge (s)
    (s) edge [loop above] node {any} (s)
    (s) edge [swap, bend right] node {$(1, \{b, d\});(1, \{b, c, d\})$} (c1)
    (c1) edge [swap] node {$(1, \{b\}); (1, \{b, c\})$} (cx)
    (cx) edge [loop left, below, rotate=45] node [rotate=-45] {$(1, \varnothing); (1, \{c\})$} (cx)
    (a) edge [loop below] node {any} (a)
    (cx) edge [bend right, swap] node {$(1, \{d\}); (1, \{c, d\})$} (a)
    (c1) edge [bend right, swap, below, rotate=-60] node {$(1, \{b,d\}); (1, \{b, c, d\})$} (a)
    (c1) edge node {$(1, \{d\}); (1, \{c, d\})$} (cy)
    (cy) edge [loop right] node {$(1, \varnothing); (1, \{c\})$} (cy)
    (cy) edge node {$(1, \{b\}); (1, \{b, c\})$} (a)
    (s) edge [swap] node {$(1, \{b\}); (1, \{b, c\})$} (x)
    (x) edge [loop left] node {$(1, \varnothing); (1, \{c\})$} (x)
    (x) edge [swap] node {$(1, \{d\}); (1, \{c, d\})$} (xy1)
    (xy1) edge [bend left] node {$(1, \{b, d\});(1, \{b, c, d\})$} (a)
    (s) edge [bend left, above, rotate=-60] node {$(1, \{b, d\}); (1, \{b, c, d\})$} (xy1)
    (s) edge [bend left] node {$(1, \{d\}); (1, \{c, d\})$} (y)
    (y) edge [loop right, above, rotate=45] node [rotate=-45] {$(1, \varnothing); (1, \{c\})$} (y)
    (y) edge node {$(1, \{b\}); (1, \{b, c\})$} (xy1)
    ;
    \end{tikzpicture}
    \caption{One possibility for $M_\phi$ for the example $\phi$. The automaton $M_D$ in Figure~\ref{fig:dataauto} roughly corresponds to the states $s_{\text{start}},s_{p_1}^{(b)},s_{p_1}^{\text{right}},$ and $s_{\text{accept}}$ in this automaton, aside from the fact that we have simplified the automaton to remove some unnecessary states and there are some more transitions between these states corresponding to other data (for example, making $x_1$ be nonadjacent to $c$ instead, which corresponds to the transition $(1, \{b,d\})$ from $s_{p_1}^{\text{right}}$ to $s_{\text{accept}}$).}
    \label{fig:phiauto}
\end{figure}
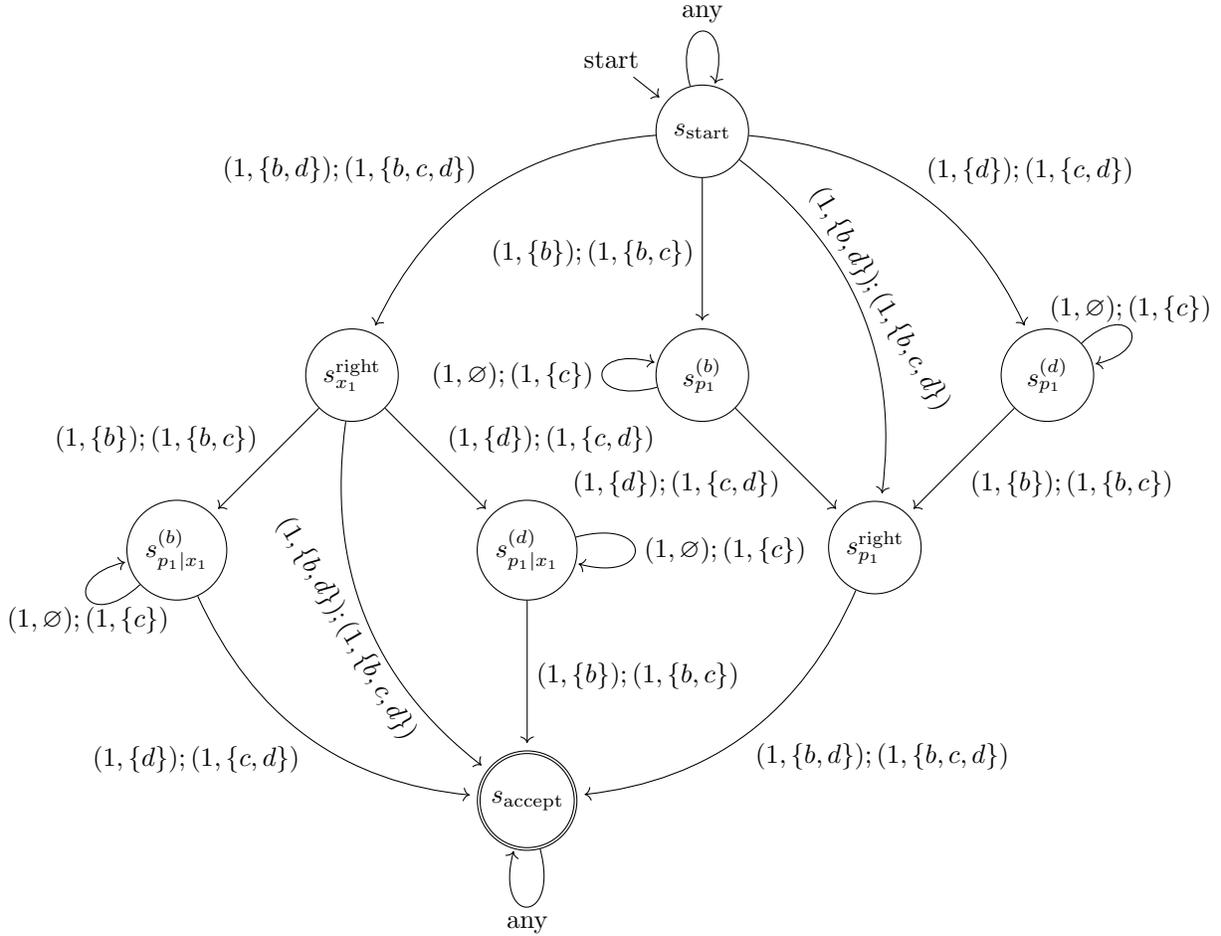

\subsection{Completing the proof}

Note that the $M_\phi$ constructed in the previous subsection do not need to have any particular behavior on strings that are not the determination string of any realization of $\Hf$. In the example at the end of the previous subsection, $M_\phi$ accepts the string
\[ (1, \{b, d\})(1, \{b, c\})(1, \{d\}) \]
and rejects the string
\[ (1, \{b\})(1, \{c\})(1, \{b\}) \]
despite the fact that these strings do not start with a vertex adjacent to $a$, do not end with a vertex adjacent to $c$, and have a vertex adjacent to $c$ somewhere in the middle. Now let us complete the proof of Theorem~\ref{thm:automaton}. The only major step left is to deal with these ``ill-formed'' strings.

\begin{proof}[Proof of Theorem~\ref{thm:automaton}]
Let the urpaths of $\Hf$ be $u_1, \dots, u_K$ with endpoints $(a_1, b_1), \dots, (a_K, b_K)$. Say a string over $\Sigma_2$ is \textit{ill-formed} if any of the following hold:
\begin{itemize}
    \item There is any symbol of the form $(i, \lambda)$ occurring after a symbol of the form $(j, \lambda')$ with $j<i$.
    \item For some $i$, there is no symbol of the form $(i, \lambda)$.
    \item For some $i$, the first symbol of the form $(i, X)$ has $a_i\not\in X$ or the last symbol of the form $(i, X)$ has $b_i\not\in X$.
    \item For some $i$, some symbol of the form $(i, X)$ has $a_i\in X$ or $b_i\in X$, other than the first symbol or last symbol of the form $(i, X)$, respectively.
\end{itemize}
It is easy to see that the language $L_I$ of ill-formed strings is a regular language, and if $\sigma$ is not ill-formed, it is the determination string of a realization of $\Hf$.

Consider the regular language
\[ L=L_I\cup\bigcup_{\Ff\in \Fc}\bigcup_{\substack{\text{partial pathograph inclusion} \\ \phi:\Ff\to H\cup\{\ud\}}} L_\phi. \]
Then $L$ contains precisely the ill-formed strings and the determination strings of realizations of $\Hf$ that are not $\Fc$-free. Let $M$ be a deterministic finite automaton recognizing the complement of $L$ (which is also a regular language). This $M$ accepts precisely the determination strings of $\Fc$-free realizations of $\Hf$, as desired.
\end{proof}

\end{document}